\documentclass{amsart}
\usepackage{amssymb,latexsym}
\usepackage{enumerate}

\newcommand\datver[1]{\def\datverp
{\par\boxed{\boxed{\text{Version: #1; Run: \today}}}}}\datver{0.1}

\newcommand\lp{`}
\newcommand\rp{'}

\newcommand\dlp{``}
\newcommand\drp{''}

\newcommand{\mfk}{\mathfrak}
\newcommand\mfkg{\mfk{g}}
\newcommand{\supp}{\operatorname{supp}}

\newcommand{\Spec}{\operatorname{Spec}}
\newcommand{\Ind}{\operatorname{Ind}}

\newcommand{\CC}{\mathbb C}

\newcommand{\NN}{\mathbb N}

\newcommand{\RR}{\mathbb R}

\newcommand\Prim{\operatorname{Prim}}

\newcommand{\CIc}{{\mathcal C}^{\infty}_{\text{c}}}

\newcommand{\maC}{\mathcal C}

\newcommand{\maD}{\mathcal D}

\newcommand{\maF}{\mathcal F}
\newcommand{\maG}{\mathcal G}
\newcommand{\maH}{\mathcal H}
\newcommand{\maI}{\mathcal I}
\newcommand{\maK}{\mathcal K}
\newcommand{\maL}{\mathcal L}

\newcommand{\maR}{\mathcal R}

\newcommand{\maT}{\mathcal T}

\newtheorem{theorem}{Theorem}[section]
\newtheorem{proposition}[theorem]{Proposition}
\newtheorem{corollary}[theorem]{Corollary}

\newtheorem{lemma}[theorem]{Lemma}
\newtheorem{notation}[theorem]{Notations}
\theoremstyle{definition}
\newtheorem{definition}[theorem]{Definition}
\theoremstyle{remark}
\newtheorem{remark}[theorem]{Remark}
\newtheorem{example}[theorem]{Example}

\newcommand{\ev}{\operatorname{ev}}

\author[V. Nistor]{Victor Nistor} \address{Universit\'{e} de Lorraine,
  UFR MIM, Ile du Saulcy, CS 50128, 57045 METZ, France
}
\email{victor.nistor@univ-lorraine.fr} 

\author[N. Prudhon]{Nicolas Prudhon} \address{Universit\'{e} de
  Lorraine, UFR MIM, Ile du Saulcy, CS 50128 57045 METZ, France}
\email{nicolas.prudhon@univ-lorraine.fr}

\thanks{ The authors were partially supported by ANR-14-CE25-0012-01
  (SINGSTAR).\\
Manuscripts available from 
{\bf 
http:{\scriptsize //}iecl.univ-lorraine.fr{\scriptsize /}\,$\tilde { }$\,Victor.Nistor{\scriptsize /}}. 
}

\date\today
\begin{document}


\title[Exhausting families]{Exhausting families of representations and
  spectra of pseudodifferential operators}

\begin{abstract} 
Families of representations of suitable Banach algebras provide a
powerful tool in the study of the spectral theory of
(pseudo)differential operators and of their Fredholmness. We introduce
the new concept of an exhausting family of representations of a
$C\sp{\ast}$-algebra $A$.  An {\em exhausting family} of
representations of a $C\sp{\ast}$-algebra $A$ is a set $\maF$ of
representations of $A$ with the property that every irreducible
representation of $A$ is weakly contained in some $\phi \in \maF$.  An
exhausting family $\maF$ of representations of $A$ has the property
that \dlp $a \in A$ is invertible if, and if, $\phi(a)$ is invertible
for any $\phi \in \maF$.\drp\ Consequently, the spectrum of $a$ is
given by $\Spec(a) = \cup_{\phi \in \maF} \Spec(\phi(a))$. In other
words, every exhausting family of representations is invertibility
sufficient, a concept introduced by Roch in {\em Algebras of
  approximation sequences: structure of fractal
  algebras} (2003). We prove several properties of exhausting
families and we provide necessary and sufficient conditions for a
family of representations to be exhausting. Using results of Ionescu
and Williams (Indiana Univ. Math. J. 2009), we show that the regular
representations of amenable, second countable, locally compact
groupoids with a Haar system form an exhausting family of
representations. If $A$ is a separable $C\sp{\ast}$-algebra, we show
that a family $\maF$ of representations of $A$ is exhausting if, and
only if, it is invertibility sufficient. However, this result is not
true, in general, for non-separable $C\sp{\ast}$-algebras.  With an
eye towards applications, we extend our results to the case of
unbounded operators. A typical application of our results is to
parametric families of differential operators arising in the analysis
on manifolds with corners, in which case we recover the fact that a
parametric operator $P$ is invertible if, and only if, its Mellin
transform $\hat P(\tau)$ is invertible, for all $\tau \in
\RR\sp{n}$. In view of possible applications, we have tried to make
this paper accessible to non-specialists in $C\sp{\ast}$-algebras.
\end{abstract}

\maketitle

\tableofcontents

\section*{Introduction}

A typical result in spectral theory of $N$-body Hamiltonians
\cite{DamakGeorgescu, GeorgescuIftimovici, GeorgescuNistor2, Simon,
  MantoiuReine} associates to the Hamiltonian $H$ a family of other
operators $H_{\phi}$, $\phi \in \maF$, such that the essential
spectrum $\Spec_{ess}(H)$ of $H$ is obtained in terms of the usual
spectra $\Spec(H_\phi)$ of $H_{\phi}$ as the closure of the union of
the later:
\begin{equation}
	\Spec_{ess}(H) \, = \, \overline{\cup}_{\phi \in \maF} \,
        \Spec(H_{\phi}) \,.
\end{equation}
It was noticed that sometimes the closure is not necessary, and one of
the motivations of our paper is to clarify this issue. Our approach is
based on the well known fact that the operators $H_{\phi}$ are
obtained as homomorphic images (in a suitable sense) of the operator
$H$, that is $H_{\phi} = \phi(H)$, where the morphisms $\phi$ are part
of a suitable family of representations $\maF$ of a certain
$C\sp{\ast}$-algebra associated to $H$. This justifies the study of
families of representations. See for example \cite{GeorgescuNistor2}
for an illustration of this approach. As a note on our terminology,
by {\em morphism} and {\em representation} of $C\sp{\ast}$-algebras,
we shall always mean a {\em $*$-morphism}, respectively,
a {\em $*$-representation}.

Another, related, motivation comes from the characterization of
Fredholm integral operators \cite{DamakGeorgescu, NP, RochBookLimit, RochBookNGT,
  SchroheFrechet, SchSch1, SchulzeBook91}.  We are
especially interested in the approach to this question using groupoids
\cite{DLR, DebordSkandalis, LMN1, LNGeometric, Vassout}.  More
precisely, for suitable manifolds $M$ and for differential operators
$D$ on $M$ compatible with the geometry, there was devised a procedure
to associate to $M$ the following data: (i) spaces $Z_\alpha$, $\alpha
\in I$; \ (ii) groups $G_\alpha$, $\alpha \in I$; and (iii)
$G_\alpha$-invariant differential operators $D_\alpha$ acting on
$Z_\alpha \times G_\alpha$.  This data can be used to characterize the
Fredholm property of $D$ as follows. Let $m$ be the order of $D$, then
\begin{multline}\label{eq.Fredholm.c}
	D : H^s(M) \to H^{s-m}(M) \mbox{ is Fredholm}
        \ \ \Leftrightarrow \ \ D \mbox{ is elliptic and }\\
	\ D_{\alpha} \ \mbox{ is invertible for all } \alpha \in I \,.
\end{multline}
Moreover, the spaces $Z_\alpha$ and the groups $G_\alpha$ are
independent of $D$.  If $M$ is compact (without boundary), then the
index $I$ is empty (so there are no $D_\alpha$s). In general, for
non-compact manifolds, the conditions on the operators $D_\alpha$ are,
nevertheless, necessary.  The non-compact geometries to which this
characterization of Fredholm operators applies include: asymptotically
euclidean manifolds, asymptotically hyperbolic manifolds, manifolds
with poly-cylindrical ends, and many others (see \cite{frascatti,
  nistorDesing} for surveys).  Again, the operators $D_\alpha$ are
homomorphic images of the operator $D$, which leads us again to the
study of families of representations.

The results in \cite{GeorgescuIftimovici, GeorgescuNistor2, LMN1}
mentioned above are the main motivation for this work, which is a
purely theoretical one on the representation theory of
$C\sp{\ast}$-algebras, even though the applications are to spectral
theory and (pseudo)differential operators.

Our main results concern \dlp exhausting families of
representations,\drp\ a concept that we introduce and study in this
paper. To explain our results, let us discuss first the important,
related concept of an \dlp invertibility sufficient family of
representations.\drp\ Recall \cite{Roch} that {\em an invertibility
  sufficient family} of representations $\maF$ of a unital
$C\sp{\ast}$-algebra $A$ is a set of representations with the property
that $a \in A$ is invertible if, and only if, $\phi(a)$ is invertible
for all $\phi \in \maF$. This concept is directly applicable to the
problems mentioned in the beginning of this introduction. It is
equivalent to the concept of a {\em strictly norming family} of
representations \cite{ExelInvGr,Roch}, a concept that we recall in the
main body of the paper. In practice, it is not straightforward to
check that a family of representations is invertibility sufficient or
strictly norming. Motivated by this, we introduce an {\em exhausting
  family} of representations of $A$ as a set $\maF$ with the property
that every irreducible representation of $A$ is weakly contained in a
representation $\phi \in \maF$. Exhausting families of representations
turn out to have many useful properties.

Here are the contents of the sections of the paper and our main
results.  In the following section--the second section--we discuss
some results on faithful family of representations in preparation and
as motivation for the study of exhausting families of representations,
which is the main thrust of the third section. Thus, in the third
section, we discuss and prove various basic properties of exhausting
families.  We also discuss their relation with invertibility
sufficient families of representations. We prove that the
$C\sp{\ast}$-algebras of groupoids $\maG$ that satisfy the Effros-Hahn
conjecture and have amenable isotropy groups have the property that
the family of regular representations $\maR = \{\pi_x\}$ is exhausting
(here $x$ is ranging through the units of $\maG$). We notice that an
example due to Voiculescu shows that this result is not true in
general.  In the fourth section we provide a necessary and sufficient
conditions for a family of representations of $A$ to be exhausting in
terms of the topology on the primitive ideal spectrum $\Prim(A)$ of
$A$. In particular, we show that for a separable $C\sp{\ast}$-algebra,
a set of representations of $A$ is invertibility sufficient if, and
only if, it is exhausting.  We also provide an example of an
invertibility sufficient family that is not exhausting in the
non-separable case. The fifth section contains some material that
allows us to treat also unbounded operators affiliated to a
$C\sp{\ast}$-algebra. The last section--the sixth--contains a typical
application of our results to parametric families of differential
operators.  This type of operators arises in the analysis on manifolds
with corners (more precisely, in the case of manifolds with
poly-cylindrical ends). In that case, we recover the fact that an
operator compatible with the geometry is invertible if, and only if,
its Mellin transform is invertible. Due to the fact that the main
applications are to areas other than the study of $C\sp{\ast}$-algebras,
we have writen the paper with an eye towards the non-specialist in
$C\sp{\ast}$-algebras. In particular, in addition to the relevant
references, we have also included a few short
proofs of some known (or essentially known) results.

We thank V. Georgescu for useful discussions and for providing us
copies of his papers. We also thank D. and I. Belti\c{t}\u{a}, S. Baaj,
M. M\u{a}ntoiu, J. Renault, and G. Skandalis and for useful comments. The 
first named author would like to also than the Max Planck Institute for
Mathematics in Bonn, where part of this work was completed, for its
hospitality. After a first version of this work has been circulated,
we have learned of the nice paper of Exel \cite{ExelInvGr}, which has
also prompted us to change some of the terminology used in this paper.

\section{$C\sp{\ast}$-algebras and their primitive ideal spectrum}

We begin with a review of some needed general $C\sp{\ast}$-algebra
results. We recall \cite{Dixmier} that a {\em $C\sp{\ast}$-algebra} is
a complex algebra $A$ together with a conjugate linear involution $*$
and a complete norm $\| \ \|$ such that $(ab)^* = b^* a^*$, $\|ab\|
\le \|a\| \|b\|$, and $\|a^*a\| = \|a\|^2$, for all $a, b \in A$.
(The fact that $*$ is an involution means that $a^{**} = a$.) In
particular, a $C\sp{\ast}$-algebra is also a Banach algebra.  Let
$\maH$ be a Hilbert space and denote by $\maL(\maH)$ the space of
linear, bounded operators on $\maH$. One of the main reasons why
$C\sp{\ast}$-algebras are important in applications is that every
norm-closed subalgebra $A \subset \maL(\maH)$ that is also closed
under taking Hilbert space adjoints is a
$C\sp{\ast}$-algebra. Abstract $C\sp{\ast}$-algebras have many
non-trivial properties that can then be used to study the concretely
given algebra $A$. Conversely, every abstract $C\sp{\ast}$-algebra is
isometrically isomorphic to a norm closed subalgebra of $\maL(\maH)$
(the Gelfand-Naimark theorem, see \cite[theorem 2.6.1]{Dixmier}).  A
{\em representation} of a $C\sp{\ast}$-algebra $A$ on the Hilbert
space $\maH_\pi$ is a morphism $\pi : A \to \maL(\maH_\pi)$ to the
algebra of bounded operators on $\maH_\pi$. 
(Recall that, in this paper, by a morphism of $C\sp{\ast}$-algebras,
we shall always mean a {\em $*$-morphism}.)
We shall use the fact that
every morphism $\phi$ of $C\sp{\ast}$-algebras (and hence any
representation of a $C\sp{\ast}$-algebra) has norm $\|\phi\| \le
1$. Consequently, every bijective morphism of $C\sp{\ast}$-algebras is
an isometric isomorphism, and, in particular
\begin{equation}\label{eq.norm}
	\|\phi(a)\| = \|a + \ker(\phi)\|_{A/\ker(\phi)} \,.
\end{equation}

A two-sided ideal $I \subset A$ is called {\em primitive} if it is the
kernel of an irreducible representation. We shall denote by $\Prim(A)$
the set of primitive ideals of $A$. For any two-sided ideal $J \subset
A$, we have that its primitive ideal spectrum $\Prim(J)$ identifies
with the set of all the primitive ideals of $A$ {\em not} containing
the two-sided ideal $J \subset A$. It turns out then that the sets of
the form $\Prim(J)$, where $J$ ranges through the set of two-sided
ideals $J \subset A$, define a topology on $\Prim(A)$, called the {\em
  Jacobson topology} on $\Prim(A)$. If $A = C(K)$, the algebra of
continuous functions on a compact space $K$, then $K$ and $\Prim(A)$
are canonically homeomorphic.  See Example \ref{matrix} for a slightly
more involved example.

Throughout this paper, we shall denote by $A$ a generic $C\sp{\ast}$-algebra.
Also, by $\phi : A \to \maL(\maH_\phi)$ we shall denote generic
representations of $A$. For any representation $\phi$ of $A$, we
define its support, $\supp(\phi) \subset \Prim(A)$ as the complement
of $\Prim(\ker(\phi))$, that is, $\supp(\phi) := \Prim(A) \smallsetminus
\Prim(\ker(\phi))$ is the set of primitive ideals of $A$ {\em
  containing} $\ker(\phi)$.

\begin{remark}\label{rem.one}
The irreducible representations of $A$ do not form a set (there are
too many of them). The {\em unitary equivalence classes} of
irreducible representations of $A$ do form a set however, which we
shall denote by $\hat A$. By $\pi : A \to \maL(\maH_\pi)$ we shall
denote an arbitrary {\em irreducible} representation of $A$. There
exists then by definition a surjective map
\begin{equation}\label{eq.def.can}
	can: \hat A \to \Prim(A)
\end{equation}
that associates to (the class of) each irreducible representation $\pi
\in \hat A$ its kernel $\ker(\pi) $. For each $a \in A$ and each
irreducible representation $\pi$ of $A$, the algebraic properties of
$\pi(a)$ depend only on the kernel of $\pi$.  That yields a well
defined function
\begin{equation}\label{eq.can.surj}
	can: \, \hat A \, \ni \pi \, \to \, \|\pi(a)\| \in [0 , \|a\|
        ] \, ,
\end{equation}
which descends to a well defined function
\begin{equation}\label{eq.can.surj2}
	n_a \, : \, \Prim(A) \ni \pi \to \|\pi(a)\| \in [0 , \|a\| ] \ ,
        \quad n_a(\ker(\pi)) = \|\pi(a)\| \,,
\end{equation}
because if $\phi_1$ and $\phi_2$ are representations of $A$ with the
same kernel, then $\|\phi_1(a)\| = \|\phi_2(a)\|$ for all $a \in A$.
\end{remark}

A $C\sp{\ast}$-algebra is {\em type I} if, and only if, the surjection
$can : \hat A \to \Prim(A)$ of Equation \eqref{eq.def.can} is, in
fact, a bijection \cite{Dixmier} (a deep result).  Then the discussion
of Remark \ref{rem.one} becomes unnecessary and several arguments
below will be (slightly) simplified since we will not have to make
distinction between equivalence classes of irreducible representations
and their kernels. Fortunately, many (if not all) of the
$C\sp{\ast}$-algebras that arise in the study of pseudodifferential
operators and of other practical questions are type I
$C\sp{\ast}$-algebras. In spite of this, it seems unnatural at this
time to restrict our study to type I $C\sp{\ast}$-algebras.
Therefore, we will not assume that $A$ is a type I
$C\sp{\ast}$-algebra, unless this assumption is really needed.  When
$A$ is a type I $C\sp{\ast}$-algebra, we will identify $\hat A$ and
$\Prim(A)$.

We shall need the following simple (and well known) lemma
\cite{Dixmier}.

\begin{lemma} 
\label{lemma.usc}
The map $n_a: \Prim(A) \ni I \to \|a + I\|_{A/I} \in [0 , \|a\| ]$ is
lower semi-continuous, that is, the set $\{I \in \Prim(A), \, \|a +
I\|_{A/I} > t \, \}$ is open for any $t \in \RR$.
\end{lemma}

We include the simple proof for the benefit of the non-specialist.
\begin{proof}
Let us fix $t \in \RR$. Since $n_a$ takes on non-negative values, we
may assume $t \ge 0$. Let then $\chi : [0, \infty) \to [0, 1]$ be a
  continuous function that is zero on $[0, t^2]$ but is $>0$ on $(t^2,
  \infty)$ and let $b = \chi(a^*a)$, which is defined using the
  functional calculus with continuous functions. If $\phi : A \to
  \maL(\maH_\phi)$ is a representation of $A$, then we have that
  $\|\phi(a)\|^2 = \|\phi(a^*a)\| \le t^2$ if, and only if,
\begin{equation*}
	\chi(\phi(a^*a)) = \phi( \chi(a^*a)) = \phi(b) = 0 \,.
\end{equation*}
Let then $J$ be the (closed) two sided ideal generated by $b$, that
is, $J := \overline{A b A}$.  Then
\begin{multline*}
	\{I \in \Prim(A), \, \|a + I\|_{A/I} \le t \, \} \, = \, \{I
        \in \Prim(A), \, b \in I \, \} \\
        = \, \{I \in \Prim(A), \, J \subset I \, \} \, = \, \Prim(A)
        \smallsetminus \Prim(J)\,,
\end{multline*}
is hence a closed set. Consequently, $\{I \in \Prim A, \, \|a +
I\|_{A/I} > t \, \}$ is open, as claimed.
\end{proof}

\section{Faithful families}

Let $\maF$ be a set of representations of $A$. We say that the family
$\maF$ is {\em faithful} if the direct sum representation
$\rho:=\oplus_{\phi \in \maF}\, \phi$ is injective. Faithful families
of irreducible representations of a $C\sp{\ast}$-algebra $A$ were
called {\em weakly sufficient} in \cite{Roch}.
The results of this subsection are for the most part very well-known,
see for instance \cite{Roch}, but we include them for the purpose of
later reference and in order to compare them with the properties of
exhausting families and strictly norming families. We have the
following well known result that will serve us as a model for
characterization of ``strictly norming families'' of representations
in the next subsection.

\begin{proposition} \label{prop.faithful1}
Let $\maF$ be a family of representations of the $C\sp{\ast}$-algebra
$A$. The following are equivalent:
\begin{enumerate} [(i)]
	\item The family $\maF$ is faithful.
	\item The union $\cup_{\phi \in \maF}\, \supp(\phi)$ is dense
          in $\Prim(A)$.
	\item $\|a\| = \sup_{\phi \in \maF}\, \|\phi(a)\|$ for all $a
          \in A$.
\end{enumerate}
\end{proposition}

\begin{proof}
(i)$\Rightarrow$(ii). We proceed by contradiction. Let us assume that
  (i) is true, but that (ii) is not true. That is, we assume that
  $\cup_{\phi \in \maF}\, \supp(\phi)$ is not dense in $\Prim(A)$.
  Then there exists a non empty open set $\Prim(J)\subset \Prim(A)$
  that does not intersect $\cup_{\phi \in \maF}\, \supp(\phi)$, where
  $J \subset A$ is a non-trivial two-sided ideal. Then $J \neq 0$ is
  contained in the kernel of $\oplus_{\phi \in\maF}\, \phi$ and hence
  $\maF$ is not faithful. This is a contradiction, and hence (ii) must
  be true if (i) is true.\\
(ii)$\Rightarrow$(iii). For a given $a\in A$, the map sending the
  kernel $\ker \pi$ of an irreducible representation $\pi$ to
  $\|\pi(a)\|$ is a lower semi-continuous function $\Prim(A) \to [0,
    \infty)$, by Lemma \ref{lemma.usc}.  Moreover, for any $a\in A$
    there exists an irreducible representation $\pi_a$ such that
    $\|\pi_a(a)\|=\|a\|$.  Hence, for every $\epsilon > 0$, $\{\pi \in
    \Prim(A), \|\pi(a)\| > \|a\| - \epsilon \}$ is a non empty open
    set (it contains $\ker \pi_a$) and then it contains some $\pi \in
    \cup_{\phi \in \maF}\, \supp(\phi)$, since the later set was
    assumed to be dense in $\Prim(A)$.  Let $\phi \in \maF$ be such
    that $\ker(\pi) \supset \ker(\phi)$. Then
\begin{equation*}
	\|a\| \ge \|\phi(a)\| \ge \|\pi(a)\| > \|a\|-\varepsilon \,,
\end{equation*}
where the first inequality is due to the general fact that
representations of $C\sp{\ast}$-algebras have norm $\le 1$ and the
second one is due to the fact that
\begin{equation*}
	\|\phi(a)\|  =  \|a + \ker(\phi)\|_{A/\ker(\phi)} 
	\ge   \|a + \ker(\pi)\|_{A/\ker(\pi)} = \|\pi(a)\| \,,
\end{equation*}
by Equation \eqref{eq.norm}.  Consequently, $\|a\| = \sup_{\phi \in
  \maF}\, \|\phi(a)\|$, as desired.\\
(iii)$\Rightarrow$(i).  Let $\rho := \oplus_{\phi \in \maF}\, \phi : A
\to \oplus_{\phi \in \maF}\, \mathcal{L}(H_\phi)$. We need to show
that $\rho$ is injective. The norm on $\oplus_{\phi \in \maF}\,
\mathcal{L}(H_\phi)$ is the sup norm, that is, $\|(T_\phi)_{\phi \in
  \maF}\| = \sup_{\phi \in \maF}\, \|T_\phi\|$.  Therefore $\|\rho(a)
\| = \sup_{\phi \in \maF} \|\phi(a)\| = \|a\|$, since we are assuming
(iii). Consequently, $\rho$ is isometric, and hence it is injective.
\end{proof}

In the next proposition we shall need to assume that $A$ is unital
(that is, that it has a unit $1 \in A$). This assumption is not very
restrictive since, given any non-unital $C\sp{\ast}$-algebra $A_0$,
the algebra with adjoint unit $A = A_0^+ := A_0 \oplus \CC$ has a
unique $C\sp{\ast}$-algebra norm. For any unital $C\sp{\ast}$-algebra
$A$ and any $a \in A$, we denote by $\Spec_A(a)$ the {\em spectrum} of
$a$ in $A$, defined by
\begin{equation*}
	\Spec_A(a) \, := \, \{\, \lambda \in\CC,\ \lambda - a \ \mbox{
          is not invertible in } \ A \, \}\,.
\end{equation*}
 is known that $\Spec_A(a)$ is, in fact, independent of the
 $C\sp{\ast}$-algebra $A$ \cite{Dixmier}.  (See next.) It is also
 known classically that $\Spec_A(a)$ is compact and non-empty, unlike
 in the case of unbounded operators \cite{Dixmier}.  For $A$
 non-unital, we let $\Spec(a) := \Spec_{A^+}(a)$.

We shall need the following general property of $C\sp{\ast}$-algebras
\cite{Dixmier}.

\begin{lemma}\label{lemma.general}
Let $A_1 \subset B$ be two $C\sp{\ast}$-algebras and $a \in A_1$ be
such that it has an inverse in $B$, denoted $a^{-1}$. Then $a^{-1} \in
A_1$.  In particular, the spectrum of $a$ is independent of the
$C\sp{\ast}$-algebra in which we compute it:
\begin{equation}\label{eq.spectrum}
	\Spec_{A_1}(a) \,  = \, \Spec_{B}(a) \, =: \, \Spec(a) \,.
\end{equation}
\end{lemma}

We shall need the following remark on extensions of representations.

\begin{remark}\label{rem.ext}
Let $B$ be a $C\sp{\ast}$-algebra and $I \subset B$ be a closed
two-sided ideal. Recall from Proposition~2.10.4 in \cite{Dixmier} that
any representation $\pi : I \to \maL(\maH)$ extends to a unique
representation $\pi : B \to \maL(\maK) \subset \maL(\maH)$, $\maK =
\overline{\pi(I)\maH}$ (the closure is actually not needed by the
Cohen-Hewitt factorization theorem).  This extension is an instance of
the Rieffel induction \cite{rieffelInducedCstar} corresponding to $I$,
regarded as an $A$--$I$ bimodule.
\end{remark}

In particular, we shall use this remark in order to deal with non-unital
algebras as follows.

\begin{notation}\label{not.+}
{\normalfont Let $I$ be a $C\sp{\ast}$-algebra and let us denote by
  $I\rp := I$ if $I$ has a unit and by $I\rp := I\sp{+} := I \oplus
  \CC$ if $I$ does not have a unit.  Let $\chi_0 : I\sp{+} \to \CC$ be
  the canonical projection.  Then, if $\maF$ is a set of
  representations of $I$, we let $\maF\rp := \maF$ if $I$ has a unit
  and $\maF\rp := \maF \cup \{\chi_0\}$ if $I$ does not have a
  unit. By implicitly extending the representations of $I$ to $I\rp$,
  we have that $\maF\rp$ is a set of representations of $I\sp{+}$.}
\end{notation}

Using this notation, we have the following result.

\begin{proposition}\label{prop.faithful2}
Let $\maF$ be a faithful family of nondegenerate representations of a
$C\sp{\ast}$-algebra $A$. An element $a \in A\rp$ is invertible if,
and only if, $\phi(a)$ is invertible in $\maL(\maH_\phi)$ for all
$\phi \in \maF\rp$ and the set $\{ \|\phi(a)^{-1} \|, \phi \in \maF\rp
\}$ is bounded.
\end{proposition}

\begin{proof} By replacing $A$ with $A\rp$, we may assume that
$A$ is unital. Since each $\phi \in \maF$ is nondegenerate, if $a$ is
  invertible, $\phi(a)$ also is invertible and
  $\|\phi(a)^{-1}\|=\|\phi(a^{-1})\|\leq \|a^{-1}\|$ is hence bounded.

Conversely, let $\rho$ be the direct sum of all the representations
$\phi \in \maF$, that is,
\begin{equation}\label{eq.def.rho}
	\rho \, := \, \oplus_{\phi \in \maF} \, \phi \, : \, A \,
        \longrightarrow \, \oplus_{\phi \in \maF}\,
        \mathcal{L}(H_\phi) \, .
\end{equation}
If $\|\phi(a)\|$ is invertible for all $\phi \in \maF$ and there
exists $M$ independent of $\phi$ such that $\|\phi(a)^{-1}\|\leq M$,
then $b := (\phi(a)^{-1})_{\phi \in \maF}$ is a well defined element
in $B := \oplus_{\phi\in\maF}\maL(\maH_\phi)$ and $b$ is an inverse
for $\rho(a)$ in $B$. Let $A_1 := \rho(A)$. Then $\rho(a) \in A_1$ is
invertible in $B$. Then observe that since $\rho$ is continuous,
injective, and surjective morphism of $C\sp{\ast}$-algebras, it
defines an isomorphism of algebras $A \to A_1$. We then conclude that
$a$ is invertible in $A$ as well.
\end{proof}

The following is a converse of the above proposition.  Recall that $a
\in A$ is called {\em normal} if $aa\sp{*} = a\sp{*}a$.

\begin{proposition}\label{prop.faithful3}
Let $\maF$ be a family of representations of a unital
$C\sp{\ast}$-algebra $A$ with the following property:
\begin{quotation}
 {\em \dlp If $a \in A$ is such that $\phi(a)$ is invertible in
  $\maL(\maH_\phi)$ for all $\phi \in \maF$ and the set $\{\,
  \|\phi(a)^{-1} \|, \, \phi \in \maF \, \}$ is bounded, then $a$ is
  invertible in $A$.\drp}
\end{quotation}
Then the family $\maF$ is faithful.
\end{proposition}

\begin{proof} 
Clearly, the family $\maF$ is not empty, since otherwise all elements
of $A$ would be invertible, which is not possible.  Let us assume, by
contradiction, that the family $\maF$ is {\em not} faithful.  Then, by
Proposition \ref{prop.faithful1}(ii), there exists a non-empty open
set $V \subset \Prim(A)$ that does not intersect $\cup_{\phi \in \maF}
\, \supp(\phi)$. Let $J \subset A$, $J \neq 0$, be the (closed)
two-sided ideal corresponding to $V$, that is, $V = \Prim(J)$. Since
$\maF$ is non-empty, we have $J \neq \Prim(A)$.  Then every $\phi \in
\maF$ is such that $\phi = 0$ on $J$. Let $a \in J$, $a \neq 0$. By
replacing $a$ with $a^*a \in J$, we can assume $a \ge 0$. Let $\lambda
\in \Spec(a)$, $\lambda \neq 0$. Such a $\lambda$ exists since $a$ is
normal and non-zero. Let $c := \lambda - a$.  Then, for any $\phi \in
\maF$, $\phi(c) = \lambda \in \CC$ is invertible and $\|\phi(c)^{-1}\|
= \lambda^{-1}$ is bounded.  However, $c$ is not invertible (in any
$C\sp{\ast}$-algebra containing it) since it belongs to the
non-trivial ideal $J$.
\end{proof}

Recall that $\maC_0(X)$ is the set of continuous functions on $X$ that
have vanishing limit at infinity. Then $\maC_0(X)$ is a commutative
$C\sp{\ast}$-algebra, and all commutative $C\sp{\ast}$-algebras are of
this form.

\begin{example} \label{ex.commutative}
Let $\mu_\alpha$, $\alpha \in I$, be a family of positive, regular
Borel measures on a locally compact space $X$. Let $\phi_\alpha$ be
the corresponding multiplication representation of the
$C\sp{\ast}$-algebra $\maC_0(X) \to \maL(L^2(X, \mu_\alpha))$.  Wee
have $\supp(\phi_\alpha) = \supp(\mu_\alpha)$ and the family $\maF :=
\{ \phi_\alpha, \alpha \in I\}$ is faithful if, and only if,
$\cup_{\alpha \in I} \supp(\mu_\alpha)$ is dense in $X$. In
particular, if each $\mu_\alpha$ is the Dirac measure concentrated at
some $x_\alpha \in X$, then $\phi_\alpha(f) = f(x_\alpha) =:
\ev_{x_\alpha}(f) \in \CC$ and $\supp(\mu_\alpha) = \{x_\alpha \}$. We
shall henceforth identify $x_\alpha \in X$ with the corresponding
evaluation irreducible representation $ev_{x_\alpha}$.  Then we have
that
\begin{equation*}
 \maF = \{\ev_{x_\alpha}, \, \alpha \in I \} \ \mbox{ is faithful
 }\ \Leftrightarrow \ \{x_\alpha,\, \alpha \in I \}\ \mbox{ is dense
   in }\ X \;.
\end{equation*}
This example extends right away to $C\sp{\ast}$algebras of the form
$\maC_0(X; \maK)$ of functions with values compact operators on some
given Hilbert space.
\end{example}

We conclude our discussion of faithful families with the following
result. We denote by $\overline \cup S_\alpha := \overline{\cup_\alpha
  S_\alpha}$ the closure of the union of the family of sets
$S_\alpha$.

\begin{proposition}\label{prop.faithful4}
Let $\maF$ be a family of representations of a unital
$C\sp{\ast}$-algebra $A$. Then $\maF$ is faithful if, and only if, for
any normal $a\in A$,
\begin{equation}\label{eq.spec.eq-faithful}
	\Spec(a) = \overline{\cup}_{\phi \in \maF} \, 
	\Spec(\phi(a))\,.
\end{equation}
\end{proposition}

\begin{proof} 
Let us assume first that the family $\maF$ is faithful and that $a$ is
normal. Since we have that $\Spec(\phi_0(a)) \subset \Spec(a)$ for any
representation $\phi_0$ of $A$, it is enough to show that $\Spec(a)
\subset \overline{\cup}_{\phi \in \maF} \, \Spec(\phi(a))$.  Let us
assume the contrary and let $\lambda \in \Spec(a) \smallsetminus
\overline{\cup}_{\phi \in \maF} \, \Spec(\phi(a))$. By replacing $a$
with $a - \lambda$, we can assume that $\lambda = 0$. We thus have
that $\phi(a)$ is invertible for all $\phi \in \maF$, but $a$ is not
invertible (in $A$). Moreover, $\|\phi(a)^{-1}\| \le \delta^{-1}$,
where $\delta$ is the distance from $\lambda =0$ to the spectrum of
$\phi(a)$, by the properties of the functional calculus for normal
operators. This is however a contradiction by Proposition
\ref{prop.faithful2}, which implies that $a$ must be invertible in $A$
as well.

To prove the converse, let us assume that $\Spec(a) \subset
\overline{\cup}_{\phi \in \maF} \, \Spec(\phi(a))$, for all normal
elements $a \in A$. Let $J$ be a non-trivial (closed {\em
  selfadjoint}) two-sided ideal on which all the representations $\phi
\in \maF$ vanish. We have to show that $J=0$, which would imply that
$\maF$ is faithful. Let $a \in J$ be a normal element. Then $\Spec(a)
\subset \overline{\cup}_{\phi \in \maF} \, \Spec(\phi(a)) =
\{0\}$. Since $a$ is normal we deduce $a=0$ and hence $J$ has no
normal element other than $0$. Then, for any $a \in J$, we can write
$a=1/2(a+a^*)+1/2(a-a^*)$, the sum of two normal elements in $J$
because $J$ is selfadjoint. Therefore $1/2(a+a^*)=1/2(a-a^*)=0$, and
hence $a=0$ and $J=0$.
\end{proof}

We refer to \cite{ASkandalis2, BuneciSurvey, LNGeometric, frascatti,
  RenaultBook, Vassout} for background material on groupoids. The
following is well known, but is useful in order to set up the
terminology and to introduce some concepts to be used below.

\begin{example}\label{ex.groupoid}
Let $\maG$ be a locally compact groupoid with units $M$ and with Haar
system $(\lambda_x)$, $x \in M$. If $d : \maG \to M$ denotes the
domain map $\maG \to M$, then we denote $\maG_A := d\sp{-1}(A)$, $A
\subset M$, and $\maG_x := d\sp{-1}(x)$, $x \in M$. We recall that
$\lambda_x$ has support $\maG_x$ (and is right invariant and
continuous in a natural sense). The {\em regular representation}
$\pi_x$ of $C\sp{\ast}(\maG)$ then acts on $L\sp{2}(\maG_x,
\lambda_x)$ by left convolution. Let $\maR := \{\pi_x,\, x \in M\}$ be
the set of regular representations of $C\sp{\ast}(\maG)$, the
$C\sp{\ast}$-algebra associated to $\maG$. Let $I$ be the intersection
of all the kernels of the representations $\pi_x$. Then the set $\maR$
is a faithful set of representations of $C_r\sp{*}(\maG) \simeq
C\sp{\ast}(\maG)/I$, the reduced $C\sp{\ast}$-algebra of $\maG$. In
general, $\maR$ will not be a faithful family of representations of
$C_r\sp{*}(\maG)$, unless the canonical projection $C\sp{*}(\maG) \to
C_r\sp{\ast}(\maG)$ is an isomorphism.
\end{example}

\section{Exhausting and strictly norming  families}

Let us notice that Example \ref{ex.commutative} shows that the `sup'
in the relation $\|a\| = \sup_{\phi \in \maF} \, \|\phi(a)\|$
(Proposition \ref{prop.faithful1}) may not be attained. It also shows
that the closure of the union in Equation \eqref{eq.spec.eq-faithful}
is needed. Sometimes, in applications, one does obtain however the
stronger version of these results (that is, that the sup is attained
and that the closure is not needed), see \cite{DamakGeorgescu,
  GeorgescuNistor2}, for example. Moreover, the condition that the
norms of $\phi(a)^{-1}$ be uniformly bounded (in $\phi$) for any fixed
$a \in A$ is inconvenient and often not needed in applications. For
this reason, we introduce now a new class of sets of representations
of $A$, the class of \dlp exhausting sets of representations,\drp\ a
class that has some additional properties. The concept of an
exhausting set of representations turns out to be closely related to
the concept of an \dlp invertibility sufficient set of
representations\drp, introduced by Roch \cite{Roch}, which we discus
first.

\subsection{Invertibility sufficient sets of representations}
We now recall the concepts of invertibility sufficient and strictly
norming families of representations \cite{ExelInvGr, Roch}. See also
\cite{RochBookLimit, RochBookNGT}.

\begin{definition}[Roch] \label{def.invSufficient}
Let $\maF$ be a set of representations of a {\em unital}
$C\sp{\ast}$-algebra $A$.
\begin{enumerate}[(i)]
\item We shall say that $\maF$ is {\em invertibility sufficient} if 
\begin{equation*}
 \dlp \; a \in A \mbox{ is invertible } \Leftrightarrow\ \phi(a)
 \mbox{ is invertible for any } \phi \in \maF \,.\; \drp
\end{equation*}
\item We shall say that $\maF$ is {\em strictly norming} if, for any
  $a \in A$, there exists $\phi \in \maF$ such that $\|a\| =
  \|\phi(a)\|$.
\end{enumerate}
\end{definition}

\begin{example} \label{ex.irred}
By classical results \cite{Dixmier}, the set of all irreducible
representations of a $C\sp{\ast}$-algebra is strictly norming. A proof
of this well-known fact is contained in \cite{ExelInvGr}. See also
Theorem \ref{thm.full}.
\end{example}

The classes of invertibility sufficient and strictly norming sets of
representations actually coincide (see Theorem \ref{thm.full}
below). Before discussing that result, however, we need to extend the
above definitions to the non-unital case.

\begin{remark}\label{rem.nonunital}
Using the notation introduced in \ref{not.+}, we obtain then the following 
form of the definition of an 
invertibility sufficient family:
\begin{quotation}
 \dlp The family $\maF$ is {\em invertibility sufficient} if $1+ a \in
 A\sp{+} := A \oplus \CC$, $a\in A$, is invertible if, and only if, $1
 + \phi(a)$ is invertible for any $\phi \in \maF$.\drp
\end{quotation}
Similarly, the definition of a strictly norming family becomes:
\begin{quotation}
 \dlp $\maF$ is {\em strictly norming } if, for any $a \in A$ and
 $\lambda \in \CC$, either there exists $\phi \in \maF$ such that
 $\|\lambda + a\| = \|\lambda + \phi(a)\| $ or $\|\lambda + a \| =
 |\lambda|$.\drp
\end{quotation}
\end{remark}

The following result was proved in the unital case in \cite{Roch}.
See also \cite{ExelInvGr}.

\begin{theorem}[Roch]\label{thm.full}
Let $\maF$ be a set of non-degenerate representations of a unital
$C\sp{\ast}$-algebra $A$. Then $\maF$ is strictly norming if, and only
if, it is invertibility sufficient.
\end{theorem}

\begin{proof}
The unital case was proved already. If $A$ does not have a unit, then
we simply replace $A$ with $A\rp$ and $\maF$ with $\maF\rp$ (see the
notation introduced in \ref{not.+}) to reduce to the unital case.
\end{proof}

Clearly, an invertibility sufficient family of representations will
consist only of non-degenerate representations, but this is not true 
of a strictly norming family. 

We now give some examples of how exhausting and strictly norming sets
of representations are useful for invertibility questions. The
following characterization of Fredholm operators is a consequence of
the definitions.

\begin{corollary}\label{cor.Fredholm}
Let $1 \in A \subset \maL(\maH)$ be a sub-$C\sp*$-algebra of bounded
operators on the Hilbert space $\maH$ containing the algebra of
compact operators on $\maH$, $\maK = \maK(\maH)$. Let $\maF$ be an
invertibility sufficient family of representations of $A/\maK$. We
then have the following characterization of Fredholm operators $a \in
A$:
\begin{quotation}
  $a \in A$ is Fredholm if, and only if, $\phi(a)$ is invertible in
  for all $\phi \in \maF$.
\end{quotation}
\end{corollary}

The following proposition is the analog of Proposition
\ref{prop.faithful4} in the framework of strictly norming families.

\begin{theorem}\label{thm.full2}
Let $\maF$ be a family of representations of a unital
$C\sp{\ast}$-algebra $A$. Then $\maF$ is invertibility sufficient if,
and only if, for any $a \in A$,
\begin{equation}\label{eq.spec.eq-full}
	\Spec(a) \, = \, \cup_{\phi \in \maF} \, \Spec(\phi(a)) \,.\,
\end{equation}
\end{theorem}

\begin{proof}  
Let us assume first that the family $\maF$ is invertibility
sufficient.  We proceed in analogy with the proof of Proposition
\ref{prop.faithful4}.  Since we have that $\Spec(\phi_0(a)) \subset
\Spec(a)$ for any representation $\phi_0$ of $A$, it is enough to show
that $\Spec(a) \subset \cup_{\phi \in \maF} \, \Spec(\phi(a))$.  Let
us assume the contrary and let $\lambda \in \Spec(a) \smallsetminus
\cup_{\phi \in \maF} \, \Spec(\phi(a))$. By replacing $a$ with $a -
\lambda$, we can assume that $\lambda = 0$. We thus have that
$\phi(a)$ is invertible for all $\phi \in \maF$, but $a$ is not
invertible (in $A$), contradicting the assumption that $\maF$ is
invertibility sufficient.

To prove the converse, let us assume that $\Spec(a) \subset \cup_{\phi
  \in \maF} \, \Spec(\phi(a))$ for all $a \in A$. Let us assume that
$a \in A$ and that $\phi(a)$ is invertible for all $\phi \in
\maF$. Then $0 \notin \cup_{\phi \in \maF} \, \Spec(\phi(a))$. Since
$\Spec(a) \subset \cup_{\phi \in \maF} \, \Spec(\phi(a))$, we have
that $0 \notin \Spec(a)$, and hence $a$ is invertible. Thus the family
$\maF$ is invertibility sufficient.
\end{proof}

\subsection{Exhausting families of representations}
It is not always easy to check that a family of representations is
invertibility sufficient (or strictly norming, for that matter). For
this reason, we introduce a slightly more restrictive class of
families of representations, the class of exhausting families of
representations. It is convenient to do this for ideals first.

\begin{definition}\label{def.exhausting.ideals}
 Let $A$ be a $C\sp{\ast}$ algebra, possibly without unit, and 
 let $\maI$ a set of (closed, two-sided) ideals $I \subset A$. 
 We say that $\maI$ is {\em exhausting} if, by definition, for 
 any irreducible representation $\pi$ of $A$, there exists
 $I \in \maI$ such that $I \subset \ker(\pi)$.
\end{definition}

We shall typically work with families of representations $\maF$. 
We consider, nevertheless, the case of families of morphisms
as well. We thus have the following closely related definition.

\begin{definition} \label{def.exhausting}
 Let $\maF$ be a set of morphisms $\phi: A \to B_{\phi}$ of a 
 (not necessarily unital) $C\sp{\ast}$-algebra $A$. The algebras
 $B_\phi$ are not fixed. We shall say that $\maF$ is {\em
  exhausting} if the family of ideals $\{\ker(\phi), \phi \in \maF\}$
  is exhausting. Similarly, a set of unitary equivalence classes
  of representations $\maF$ of $A$ is exhausting if the corresponding
  set of kernels is exhausting.
\end{definition}

The following simple remark is sometimes useful.

\begin{remark} \label{rem.simple}
 Let $\phi$ be a representation of $A$. Recall that $\supp(\phi)$ is the 
 set of primitive ideals of $A$ that contain $\ker(\phi)$. 
 Moreover, $\ker(\phi)$ depends only on the unitary equivalence
 class of $\phi$. We then see that $\maF$ is exhausting if, and only if, 
 $\Prim(A) = \cup_{\phi \in \maF}\, \supp(\phi)$.
\end{remark}

Recall that we denote by $A\rp := A$ if $A$ has a unit and $A\rp :=
A\sp{+} := A \oplus \CC$, the algebra of adjoint unit, if $A$ does not
have a unit.

\begin{proposition}
Let $A$ be a possibly non-unital $C\sp{\ast}$-algebra and let $\maF$
be a family of representations of $A$. We denote by $\maF\rp = \maF$
if $A$ has a unit and by $\maF\rp := \maF \cup \{\chi_0\}$, where
$\chi_0 : A\rp = A \oplus \CC \to \CC$ is the canonical projection (as
in \ref{not.+}).  Then we have
\begin{enumerate}[(i)]
  \item $\maF$ is an exhausting set of representations of $A$ 
  if, and only if, $\maF\rp$ is an exhausting set of representations of $A\rp$.
  \item $\maF$ is an invertibility sufficient set of representations
    of $A$ if, and only if, $\maF\rp$ is an invertibility sufficient
    set of representations of $A\rp$.
\end{enumerate}
\end{proposition}

\begin{proof} 
To prove (i), we only need to consider the case when $A$ does not have
a unit. The result then follows from Remark \ref{rem.simple} and from
the relation $\Prim(A\rp) = \Prim(A\sp{+}) =
\Prim(A) \cup \{\ker(\chi_0)\}$, where, we recall, $\ker(\chi_0) =
A$. The other statement is really the corresponding definitions.
\end{proof}

\begin{remark}
 Let $\maF_i$, $i = 1,2$, be two families of representations of
 $A$. Let denote by $\maI_i := \{ \ker(\phi), \phi \in \maF_i \}$. We
 assume that $\maI_1 = \maI_2$.  Then the families $\maF_i$ are at the
 same time exhausting or not. The same is true for the properties of
 being strictly norming, or invertibility sufficient. So these
 properties are really properties of a family of ideals of $A$ rather
 than of families of representations of $A$. Nevertheless, it is
 customary to work with families of representation rather than
 families of ideals. In the same way, we can consider the analogous
 properties of families of {\em morphisms} of $C\sp{\ast}$-algebras.
\end{remark}

Let us record the following simple facts, for further use.

\begin{proposition}\label{prop.full.strictly.norming}
Let $\maF$ be a set of representations of a $C\sp{\ast}$-algebra. If
$\maF$ is exhausting, then $\maF$ is invertibility preserving and
hence also strictly norming. If $\maF$ is strictly norming, then it is
also faithful.
\end{proposition}

\begin{proof} Let $A$ be the given $C\sp{\ast}$-algebra. 
Let us prove first that any exhausting family $\maF$ is strictly
norming. Indeed, let $a \in A\rp$.  Then there exists an irreducible
representation $\pi$ of $A\rp$ such that $\|\pi(a)\| = \|a\|$
\cite{Dixmier}. Unless $1 \notin A$ and $\pi = \chi_0$, where $\chi_0
: A\rp = A \oplus \CC \to \CC$ is the projection, there will exist
$\phi \in \maF$ such that $\pi \in \supp(\phi)$. Then, as in the proof
of (ii)$\Rightarrow$(iii) in Proposition \ref{prop.faithful1}, we have
that $\|a \| = \|\pi(a)\| \le \|\phi(a)\| \le \|a\|$. Hence
$\|\phi(a)\| = \|a\|$. On the other hand, if $1 \notin A$ and $\pi =
\chi_0$, then let $a = \lambda + a_0$, with $\lambda \in \CC$ and $a_0
\in A$. Then $\|\lambda + a_0\| = \|a\| = \|\pi(a)\| = |\lambda|$.
Since any strictly norming family is invertibility preserving, by
Theorem \ref{thm.full}, the first part of the proposition follows.

Let us prove first that any strictly norming family $\maF$ is
faithful. Indeed, let us consider the representation $\rho :=
\oplus_{\phi \in \maF} \, \phi : A \, \to \, \oplus_{\phi \in \maF} \,
\mathcal{L}(H_\phi).$ By the definition of a strictly norming family
of representations, the representation $\rho$ is isometric. Therefore
it is injective and consequently $\maF$ is faithful.
\end{proof}

We summarize the above Proposition in
\begin{equation*}
  \maF \mbox{ exhausting } \Rightarrow 
  \maF \mbox{ strictly norming  } \Rightarrow 
  \maF \mbox{ faithful.} 
\end{equation*}

In the next two examples we will see that there exist faithful
families that are not strictly norming and strictly norming families
that are not exhausting.

\begin{example} \label{ex.commutative2}
We consider again the framework of Example \ref{ex.commutative} and
consider only families of {\em irreducible representations}. Thus $A =
\maC_0(X)$, for a locally compact space $X$. The irreducible
representations of $A$ then identify with the points of $X$, since $X
\simeq \Prim(A) = \hat A$. A family $\maF$ of {\em irreducible}
representations of $A$ thus identifies with a subset $\maF \subset
X$. We then have that a family $\maF \subset X$ of irreducible
representations of $A = \maC_0(X)$ is faithful if, and only if, $\maF$
is {\em dense} in $X$. On the other hand, a family of irreducible
representations of $A = \maC_0(X)$ is exhausting if, and only if,
$\maF = X$.
\end{example}

The relation between exhausting and strictly norming families is not
so simple. We begin with the following remark on the above example.

\begin{remark}\label{rem.metrisable}
If in Example \ref{ex.commutative2} $X$ is moreover {\em metrisable},
then every strictly norming family $\maF \subset X$ is also
exhausting, because for any $x \in X$, there exists a compactly
supported, continuous function $\psi_x : X \to [0, 1]$ such that
$\psi_x(x) = 1$ and $\psi_x(y) < 1$ for $y \neq x$ (we can do that by
arranging that $\psi_x(y) = 1 - d(x, y)$, for $d(x, y)$ small, and use
the Tietze extension theorem.  In general, however, it is not true
that any strictly norming family is exhausting. Indeed, let $I$ be an
uncountable set and $X = [0, 1]^I$. Let $x \in X$ be arbitrary, then
the family $\maF := X \smallsetminus \{x\}$ is strictly norming but is
not exhausting. Indeed, let $f : X \to [0, 1]$ be a continuous
function such that $f(x) = 1$.  Since $f$ depends on a countable
number of variables, the set $\{f = 1\}$ will not be reduced to $x$
alone. See also Theorem \ref{thm.2nd.count}.
\end{remark}

We conclude this subsection with the following result that
is relevant for the next subsection. See also \cite{RochBookNGT}
and the comment at the end of this subsection. The results in that
book can be used to give a quick proof of the following results for
invertibility sufficient families (which
are essentially contained in that book). For the benefit of the reader,
we include nevertheless the short, direct proofs, since we are also
interested in exhausting families.

\begin{proposition}\label{prop.ideal}
Let $I \subset A$ be an ideal of a $C\sp{\ast}$-algebra. Let $\maF_I$
be a set of nondegenerate representations of $I$ and $\maF_{A/I}$ be a
set of representations of $A/I$.  Let $\maF := \maF_I \cup
\maF_{A/I}$, regarded as a family of representations of $A$. If
$\maF_I$ and $\maF_{A/I}$ are both exhausting, then $\maF$ is also
exhausting. The same result holds by replacing exhausting with
strictly norming.
\end{proposition}

\begin{proof}
 We have that $\Prim(A)$ is the disjoint union of $\Prim(I)$ and
 $\Prim(A/I)$.  Since $\cup_{\phi \in \maF_I} \supp(\phi) \subset
 \Prim(I)$ and $\cup_{\phi \in \maF_{A/I}} \supp(\phi) \subset
 \Prim(A/I)$, the result about exhausting families follows from the
 definition.
 
 Let us assume that both $\maF_I$ and $\maF_{A/I}$ are strictly
 norming and let $a \in A$.  We may assume that $A$ is unital.  We
 want to show that $\maF$ is also strictly norming, that is, that
 there exists $\phi \in \maF_I \cup \maF_{A/I}$ such that $\|a\| =
 \|\phi(a)\|$.  By replacing $a$ with $a\sp{*}a$, we can assume that
 $a \ge 0$.  Since $\maF_{A/I}$ is strictly norming, there is $\phi
 \in \maF_{A/I}$ such that $\|a + I\|_{A/I} = \| \phi(a)\|$.  If $\|a
 + I\|_{A/I} = \|a\|$, we are done.  Otherwise, let $\psi$ be a
 continuous function on $\Spec(a)$ that is zero on $\Spec_{A/I}(a +
 I)$ and such that $\psi(\|a\|) = \|a\|$ and $\psi(t) \le t$ for $t
 \ge 0$.  Then $\psi(a) \in I$ and $\|\psi(a)\| = \|a\|$.  Since the
 family $\maF_{I}$ is strictly norming, there exists $\phi \in \maF_I$
 such that
 $$
  \|a\| \, = \, \|\psi(a)\| \, = \, \|\phi(\psi(a))\| \, = \,
 \|\psi(\phi(a))\| \, \le \, \|\phi(a)\|
 $$
This shows that the family $\maF$ is strictly norming.
\end{proof}

We have the following consequence that is sometimes useful in
applications.

\begin{corollary}\label{cor.ideal}
 Let $I \subset A$ be a two-sided ideal in a $C\sp{\ast}$-algebra $A$.
 Let $\maF$ be an invertibility preserving family of representations 
 of $I$. Then $a \in A$ is invertible if, and only if, $a$ is
 invertible in $A/I$ and $\phi(a)$ is invertible for all 
 $\phi \in \maF$.
\end{corollary}

\begin{proof}
 Since $\maF$ is an invertibility preserving set of representations of
 $I$, it consists of non-degenerate representations, which will hence
 extend uniquely to $A$. Let $\pi$ be an isometric representation of
 $A/I$. The result then follows from Proposition \ref{prop.ideal}
 applied to $\maF_I := \maF$ and $\maF_{A/I} := \{\pi\}$.
\end{proof}

Results closely related to Proposition \ref{prop.ideal} and Corollary 
\ref{cor.ideal} were obtained in \cite{RochBookNGT} under the name of 
``lifting theorems.'' See especially Section 6.3 of that book. The results in 
that book were typically obtained in a more general general setting:
often using ideals in a Banach algebra and sometimes using even general ideals (and 
morphisms). The interested reader should consult that book as well.

\subsection{Groupoid algebras and the Effros-Hahn conjecture}
We now show how one can check in the framework of locally compact
groupoids (with additional properties) that certain families of
representations are exhausting, thus generalizing some results of
\cite{ExelInvGr}.

We refer to the Example \ref{ex.groupoid} and, especially, to the
references quoted before that example, for notations and results
pertaining to groupoids. In particular, we shall denote by $d$ and $r$
the domain and range maps of a groupoid $\maG$ and by $\maG_x\sp{x} :=
d\sp{-1}(x) \cap r\sp{-1}(x)$ the {\em isotropy group} of $x$. This is
the group of arrows (or morphisms) of $\maG$ that have domain and
range equal to the unit $x$. Also, we continue to denote by $\maR :=
\{\pi_y, y \in M\}$ the set of regular representations of a locally
compact groupoid $\maG$ with Haar system and units $M$. Recall that we
denote $\maG_A := d\sp{-1}(A)$, $A \subset M$, and $\maG_x :=
d\sp{-1}(x)$, $x \in M$.

We shall say that a locally compact groupoid $\maG$ with a Haar system
{\em has the  generalized Effros-Hahn property} if every primitive ideal of
$C\sp{\ast}(\maG)$ is induced from an isotropy subgroup $\maG_y\sp{y}$
of $\maG$ \cite{ionescuWilliamsEHC, renault91}. (This should not be
confused with the variouis \dlp EH {\em induction}
properties\drp\ introduced in \cite{EchterhoffWilliams08}.)  We shall
write $\Ind_{y}\sp{\maG}(\sigma)$ for the induced representation of
$C\sp{\ast}(\maG)$ from the representation $\sigma$ of $\maG_y\sp{y}$.
If $\maG$ has the  generalized Effros-Hahn property and all the isotropy groups
$\maG_y\sp{y}$, $y \in M$ are amenable, we say that $\maG$ is {\em
  EH-amenable}.

\begin{theorem}\label{thm.sufficient}
Let $\maG$ be a locally compact groupoid with a Haar system and units
$M$. If $\maG$ is EH-amenable, then the family $\maR := \{\pi_y, y \in
M\}$ of regular representations of $C\sp{\ast}(\maG)$ is exhausting.
In particular, the family $\maR$ is invertibility sufficient and the
canonical map $C\sp{\ast}(\maG) \to C\sp{\ast}_r(\maG)$ is an
isomorphism.
\end{theorem}

\begin{proof} 
Let $I$ be any primitive ideal of $C\sp{\ast}(\maG)$.  Then $I$ is
induced from the isotropy group $\maG_y\sp{y}$, $y \in M$, by the
assumption that $\maG$ has the  generalized Effros-Hahn property.  Since
$\maG_y\sp{y}$ is amenable, every irreducible representation of
$\maG_y\sp{y}$ is weakly contained in the regular representation
$\rho_y$ of $\maG_y\sp{y}$.  But $\Ind_{y}\sp{\maG}(\rho_y)$ is the
regular representation $\pi_y$ of $C\sp{\ast}(\maG)$ on
$L\sp{2}(\maG_y)$. Since induction preserves the weak containment of
representations (see Proposition 6.26 of \cite{rieffelInducedCstar}),
we obtain that $I$ contains $\ker(\pi_y)$.  This proves that the
family $\maR := \{\pi_y, y \in M\}$ is exhausting.  Therefore $\maR$
is also faithful, and hence $C\sp{\ast}(\maG) \simeq
C_r\sp{\ast}(\maG)$ (see Example \ref{ex.groupoid}).  The family
$\maR$ is invertibility sufficient since it is exhausting (see
Proposition \ref{prop.full.strictly.norming}).
\end{proof}

We then obtain the following consequence.

\begin{theorem}\label{thm.amenable}
Let $\maG$ be a locally compact groupoid with a Haar system and units
$M$.  If $\maG$ is Hausdorff, second countable, and (topologically)
amenable, then the family $\maR := \{\pi_y, y \in M\}$ is exhausting.
\end{theorem}

\begin{proof}
Since $\maG$ is an amenable, Hausdorff, second countable, locally
compact groupoid with a Haar system, we have that $\maG$ satisfies the
Effros-Hahn conjecture by the main result in
\cite{ionescuWilliamsEHC}, that is, it has the  generalized Effros-Hahn property.
Since $\maG$ is amenable, all its isotropy groups $\maG_x\sp{x}$ are
amenable \cite{RenaultAnantBook}. The result then follows from Theorem
\ref{thm.sufficient}.
\end{proof}

This result extends a result of \cite{ExelInvGr}, who considered the
case of etale groupoids. Let $\maG$ be a locally compact groupoid with
a Haar system and units $M$. We notice, however, that the family $\maR
:= \{\pi_y, y \in M \}$ of regular representations of the reduced
$C\sp{\ast}$-algebra $C_r\sp{\ast}(\maG)$ of $\maG$ is not exhaustive
in general, as can be seen from the following example.

\begin{remark}
 Let $G$ be the free group on two generators and let $K_n \subset G$,
 $n \in \NN$, be decreasing sequence of normal subgroups of $G$ of
 finite index with $\cap_{n=1}\sp{\infty} K_n = \{ 1 \}$. Let us
 consider the family of groups $\maG := \cup_n \{n\} \times G/K_n$, with
 $n \in \NN \cup \{\infty\}$ and $K_{\infty} := \{ 1\}$. It is a
 groupoid with units $\NN \cup \{\infty\}$. Its domain and range map
 are equal and equal to the projection onto the first component. The
 topology on $\maG_{\NN} := d\sp{-1}(\NN)$, the restriction of $\maG$
 to $\NN$, is discrete. A basis of the system of neighborhoods of
 $(\infty, g)$ is given by the sets $\{(n, gK_n), n \ge N\}$, where $N
 \ge 1$ is arbitrary ($g \in G$). We have that the trivial
 representation of $G$ defines a representation $\chi$ of
 $C\sp{\ast}(\maG)$ supported at $\{\infty\}$.  The trivial
 representation of $G$ is the limit of the trivial representations of
 $G/K_n$, so it descends to a representation of
 $C_r\sp{\ast}(\maG)$. However, the trivial representation of $G$ is
 not contained in the support of any of the representations
 $\lambda_n$, $n \in \NN \cup \{\infty\}$, since $G$ is not amenable.
 Thus the family of regular representations $\lambda_n$, $n \in \NN
 \cup \{\infty\}$ is not exhaustive.  This example is due to
 Voiculescu and it answers (in the negative) a question of Exel
 \cite{ExelInvGr}.
\end{remark}

We are ready to prove now that the class of EH-amenable groupoids is
closed under extensions and that suitable ideals and quotients of
EH-amenable groupoids are also EH-amenable.

\begin{proposition}\label{prop.EHexact}
Let $\maG$ be a locally compact groupoid with a Haar system and units
$M$.  Let $U \subset M$ be an open $\maG$-invariant subset and $F := M
\smallsetminus U$.  We have that $\maG$ is EH-amenable if, and only
if, $\maG_F$ and $\maG_U$ are EH-amenable.
\end{proposition}

\begin{proof}
It is clear that the isotropy groups $\maG_x\sp{x}$ of $\maG$ are
given by the isotropy groups of the restrictions $\maG_F$ and
$\maG_U$. This gives that all the isotropy groups of $\maG$ are
amenable if, and only if, the same property is shared by all the
isotropy groups of the restrictions $\maG_F$ and $\maG_U$.

Let us turn now to proving the induction property for the primitive
ideals.  We need the following general fact.  Let $A$ be a
$C\sp{\ast}$-algebra and $J \subset A$ be a two-sided ideal, then we
have that $\Prim(A)$ is the disjoint union of $\Prim(J)$ and
$\Prim(A/J)$ \cite{Dixmier}. This correspondence sends a primitive
ideal $I$ of $A$ to $I \cap J$, if $I \cap J \neq J$, and otherwise
(i.e. if $J \subset I$) it sends $I$ to $I/J$, which is an ideal of
$A/J$.

We shall use this correspondence as follows.  Let $I$ be primitive
ideal of $C\sp{\ast}(\maG)$. Since $C\sp{\ast}(\maG_U)$ is an ideal of
$C\sp{\ast}(\maG)$ and $C\sp{\ast}(\maG)/C\sp{\ast}(\maG_U) \simeq
C\sp{\ast}(\maG_F)$, by a result of Renault \cite{RenaultBook,
  renault91}, we have that $I$ corresponds uniquely to either a
primitive ideal of $C\sp{\ast}(\maG_F)$ or to a primitive ideal of
$C\sp{\ast}(\maG_U)$. We shall consider these two cases
separately. Anticipating, the first case will correspond to induced
representations from isotropy groups $\maG_y\sp{y}$ with $y \in F := M
\smallsetminus U$ and the second case will correspond to induced
representations from isotropy groups $\maG_y\sp{y}$ with $y \in U$.
We first notice that the restriction of the induced representation
$\Ind_{y}\sp{\maG}(\sigma)$ of $C\sp{\ast}(\maG)$ (induced from the
representation $\sigma$ of $\maG_y\sp{y}$) restricts to a non-zero
representation of $C\sp{\ast}(\maG_U)$ if, and only if, $y \in U$.

Let us then consider a primitive ideal $I \supset C\sp{\ast}(\maG_U)$
of $C\sp{\ast}(\maG)$ and $I/C\sp{\ast}(\maG_U)$ the corresponding
ideal of $C\sp{\ast}(\maG_F) \simeq
C\sp{\ast}(\maG)/C\sp{\ast}(\maG_U)$. Then $I$ is induced from the
irreducible representation $\sigma$ of $\maG_y\sp{y}$ if, and only if,
$y \in F$ and $I/C\sp{\ast}(\maG_U)$ is induced from the irreducible
representation $\sigma$ of $\maG_y\sp{y}$.  This follows directly from
the definition of induced representations \cite{rieffelInducedCstar};
in fact, the inducing module is the same for both ideals.

On the other hand, if the primitive ideal $I$ of $C\sp{\ast}(\maG)$
does not contain $C\sp{\ast}(\maG_U)$, then again we notice that $I$
is induced from the irreducible representation $\sigma$ of
$\maG_y\sp{y}$ if, and only if, $y \in U$ and $I \cap
C\sp{\ast}(\maG_U)$ is induced from the irreducible representation
$\sigma$ of $\maG_y\sp{y}$. This again follows from the results in
\cite{rieffelInducedCstar}, more precisely, from Induction in Stages
Theorem 5.9 of that paper. Indeed, extending non-degenerate
representations of an ideal to the whole algebra is a particular case
of induction in stages (see the Remark \ref{rem.ext}). The inductions
modules are again the same.
\end{proof}

\section{Topology on the spectrum and strictly norming  families}

Let us discuss now in more detail the relation between the concept of
invertibility sufficient family and the simpler (to check) concept of
an exhausting family. 
The following theorem studies $C\sp{\ast}$-algebras with the 
property that every invertibility sufficient family
is also exhausting.
It explains Example \ref{ex.commutative2} and
Remark \ref{rem.metrisable}.

\begin{lemma}\label{lemma.preliminary} 
Let $A$ be a $C\sp{\ast}$-algebra, $J$ a two-sided ideal, and $\pi$ a
representation of $A$ such that $\pi$ is nondegenerate on $J$. Also let
$a\in A$, $0\leq a \leq 1$, such that $\|\pi(a)\|=1$ and choose $\eta>0$. Then
there exists $c\in J$, $c \ge 0$, $\|c\|\leq \eta$ such that
$\|\pi(a+c)\|\geq 1+\eta/2$.
\end{lemma}

\begin{proof}
For any fixed $\varepsilon >0$ there exists a unit vector $\xi$ such
that $\langle \pi(a)\xi,\xi \rangle \geq 1-\varepsilon$. Let us
consider then the positive linear form $\varphi : A \to \CC$ defined
by $\varphi(b)\colon =\langle \pi(b)\xi,\xi \rangle$. 
%
%
If $(u_\lambda)$ is an approximate unit in $J$, then
\begin{equation*}
  \|\varphi\| \geq\| \varphi|_J \| \, = \, \lim \varphi(u_\lambda) \,
  = \, \| \xi \| =1 \;.
\end{equation*}
So $\|\varphi|_J \|= \|\phi\| = 1$. Hence there exists $c_0 \in J$,
$c_0 \ge 0$, $\|c_0\| = 1$, such that $\varphi(c_0)\geq
1-\varepsilon$. We then set $c=\eta c_0$ and indeed, for $\varepsilon$
small enough
\begin{equation*}
  \|a+c\| \geq \varphi(a+c)\, \geq \,
  1-\varepsilon+\eta(1-\varepsilon) \geq 1+\eta/2 \;.
\end{equation*}
This completes the proof.
\end{proof}

We shall use the above lemma in the form of the following
corollary.

\begin{corollary}\label{cor.countable}
Let $\pi_0$ be an irreducible representation of a $C\sp{\ast}$-algebra
$A$ and let $I_0 := \ker (\pi_0) \in \Prim(A)$. We assume that we
are given  decreasing sequence $V_0 \supset \ldots \supset V_n \supset V_{n+1} \ldots $ 
of open neighborhoods of $I_0$ in $\Prim(A)$. Then there
exists $a \in I_0$ such that $\|a\|=\|\pi_0(a)\| =1$ and $\| \pi(a)\|
\le 1 - 2^{k}$ for any irreducible representation $\pi$ such that
$\ker(\pi) \notin V_k$.
\end{corollary}

\begin{proof}
To construct $a \in A$ with the desired properties, let us consider
the ideals $J_n$ defining the sets $V_n$, that is, $V_n =
\Prim(J_n)$, $n \ge 0$. Since $V_n \subset V_{n-1}$ for all $n$,
we have that $J_n \subset J_{n-1}$ for all $n$.

The element $a$ we are looking for will be the limit of a
sequence $(a_n)$, $a_n \in A$, where the $a_n$ are defined 
inductively to satisfy the following properties:
\begin{enumerate}[(i)]
\item $0 \le a_n \le 1$;
\item $\|\pi_0(a_n)\| = 1$;
\item $\|\pi(a_n)\| \le 1 - 2^{-k}$ for all irreducible
  representations $\pi$ such that $\ker(\pi) \in \Prim(A)
  \smallsetminus \Prim(J_k)$ for $k=0, 1, \ldots, n$;
\item $\|a_n - a_{n-1}\| \le 2^{2-n}$ for $n \ge 1$.
\end{enumerate}

We define the initial term $a_0$ as follows. We first choose $b_0 \in J_0$ 
such that $0 \le b_0$, and $\pi_0(b_0) \neq 0$. By rescaling $b_0$ with a 
positive factor, we can assume that $\|\pi_0(b_0)\| = 1$. Let then $\chi_0 :
[0, \infty) \to [0, 1]$ be the continuous function defined by
  $\chi_0(t) = t$ for $t \le 1$ and $\chi_0(t) = 1$ for $t \ge
  1$. Then we define $a_0 = \chi_0(b_0)$. Conditions (i--iv) are then
  satisfied

Next, $a_{n}$ is defined in terms $a_{n-1}$. In order to do that,  
we first define auxiliary elements $c_n$ and $b_n =
a_{n-1} + c_n$ as follows. By Lemma \ref{lemma.preliminary}, there exists
$c_n \in J_n$, $c_n \ge 0$, $\|c_n\| \le 2^{1-n}$, such that
$\|\pi_0(b_n) \| \ge 1 + 2^{-n}$.  Let then $\chi_n : [0, \infty) \to
  [0, 1]$ be the continuous function defined by $\chi_n(t) = t$ for $t
  \le 1-2^{1-n}$, $\chi_n$ linear on $[1-2^{1-n}, 1]$ and on $[1, 1 +
    2^{-n}]$, $\chi_n(1) = 1 - 2^{-n}$, and $\chi_n(t) = 1$ for $t \ge 1
  + 2^{-n}$. Then we define $a_n = \chi_n(b_n)$.
\medskip

\noindent {\em Claim.}  The sequence $a_n \in A$ just constructed
satisfies conditions (i--iv). 
\medskip

\noindent Indeed, we have checked our conditions for $n=0$,
so let us assume $n \ge 1$ and check our conditions for $a_n \in A$ one by one:
\medskip

\noindent (i) We have that $a_{n-1}, c_n \ge 0$, hence $b_n := a_{n-1}
+ c_n \ge 0$.  Since $0 \le \chi_n \le 1$, we obtain that $0 \le a_n
:= \chi_n(b_n) \le 1$.
\medskip

\noindent (ii) Since $0 \le \chi_n \le 1$,
$\chi_n(t) = 1$ for $t \ge 1 + 2^{-n}$, and 
$\|\pi_0(b_n) \| \ge 1 + 2^{-n}$, we have that $\|\pi_0(a_n)\| =
\|\pi_0(\chi_n(b_n)) \| = \| \chi_n(\pi_0(b_n)) \| = 1$.
\medskip

\noindent (iii) Let $\pi \in \hat A$ be such that $\ker(\pi) \in
\Prim(J_k)^c := \Prim(A) \smallsetminus \Prim(J_k)$, for some $k$, 
$0 \le k \le n$.  We need to check that $\|\pi(a_n)\| \le 1 - 2^{-k}$.

We have tat $\pi$ vanishes on $J_k$, and hence
$\pi(c_n) = 0$ since $c_n \in J_n \subset J_k$, $k \le n$. Therefore,
\begin{equation*}
	\pi(a_n) = \pi(\chi_n(b_n)) = \chi_n(\pi(b_n)) =
        \chi_n(\pi(a_{n-1})) \,.
\end{equation*}
We shall consider now two cases: $k < n$ and 
$k = n$. 
\medskip

\noindent {\em Case 1.}\
If $k < n$, then $\|\pi(a_{n-1}) \| \le 1 - 2^{-k} \le 1 - 2^{1-n}$,
by the induction hypothesis.
Since $\chi_n(t) = t$ for $t \le 1 - 2^{1-n}$, we obtain 
$\pi(a_n) = \chi_n(\pi(a_{n-1})) = \pi(a_{n-1})$, and hence
$\| \pi(a_n) \| = \| \pi(a_{n-1}) \| \le 1 - 2^{-k}$ for $k <n$. 
\smallskip

\noindent {\em Case 2.}\ If $k = n$,  we have
$\|\pi(a_n) \| = \|\chi_n(\pi(a_{n-1}))\| \le 1 - 2^{-n} = 1 - 2\sp{-k}$,
since $\pi(a_n) = \chi_n(\pi(a_{n-1}))$, $\chi_n
(t) \le 1 - 2^{-n}$ for $t \le 1$, and $0 \le a_{n-1} \le 1$.
\medskip

\noindent (iv) We have $\|b_n\| \le \|a_{n-1}\| + \| c_n\| \le 1 +
2^{1-n}$. Since $|\chi_n(t) - t | \le 2^{1 -n}$ for all $t \le 2\sp{1-n}$, we have
$\|a_{n} - b_n\| \le 2^{1-n}$. Hence
\begin{equation*}
  \|a_n - a_{n-1}\|  \le \| a_n -
  b_{n} \| + \| b_n - a_{n-1} \| \\
  \le 2^{1 -n} + \|c_n\| \le 2^{2 -n}.
\end{equation*}
This completes the proof of our claim, and hence the sequence 
$a_n$ constructed above satisfies Conditions (i-iv).
\medskip

Let us now show how to use the fact that the sequence $a_n \in A$
satisfies Conditions (i-iv) to construct $a$ as in the statement of
this corollary. First of all, Condition (iv) allows us to define $a :=
\lim_{n \to \infty} a_n$. Let us show that $a \in A$ satisfied the
desired conditions. Since Conditions (i--iii) are compatible with
limits, we have
\begin{enumerate}[(i)]
\item $0 \le a \le 1$;
\item $\|\pi_0(a)\| = 1$;
\item $\|\pi(a)\| \le 1 - 2^{-k}$ for all irreducible representations
$\pi$ such that $\ker(\pi) \in \Prim(A) \smallsetminus
  \Prim(J_k)$ for $k \ge 0$.
\end{enumerate}

Thus $a$ has the properties stated in this corollary, which completes
the proof.
\end{proof}

\begin{proposition}\label{prop.2ndcount} 
Let $A$ be a unital $C\sp{\ast}$-algebra. Let us assume that every $I
\in \Prim(A)$ has a countable base for its system of
neighborhoods. Then every strictly norming family $\maF$ of
representations of $A$ is also exhausting.

Let us assume that $\Prim(A)$ is a $T_1$ space. Then the converse
is also true, that is, if every strictly norming family $\maF$ of representations
of $A$ is also exhausting, then every $I \in \Prim(A)$ has a countable
base for its system of neighborhoods.
\end{proposition}

We think that the condition that $\Prim(A)$ be $T_1$ is not necessary.
However, as noticed by Roch, the proof below requires this assumption.

\begin{proof} 
Let us prove first the first part of the statement, so let us assume
that every primitive ideal $I \in \Prim(A)$ has a countable base for
its system of neighborhoods and let $\maF$ be a strictly norming
family of representations of $A$. We need to show that $\maF$ is
exhausting.  We shall proceed by contradiction. Thus, let us assume
that the family $\maF$ is not exhausting. Then there exists a
primitive ideal $I_0 = \ker (\pi_0) \in \Prim(A) \smallsetminus
\cup_{\phi \in \maF} \, \supp(\phi)$. Let
\begin{equation*}
	V_0 \supset \ldots \supset V_n \supset V_{n+1} \ldots \supset
        \{I_0\} = \cap_k V_k
\end{equation*}
be a basis for the system of neighborhoods of $I_0$ in $\Prim(A)$. We
may assume without loss of generality that that the neighborhoods
$V_n$ consist of open sets. 
Corollary \ref{cor.countable} then yields $a \in A$ such that
$\|a\|=\|\pi_0(a)\| =1$, but $\| \pi(a)\| \le 1 - 2^{k}$ for any
irreducible representation $\pi$ of $A$ such that $\ker(\pi) \in \Prim(A)
\smallsetminus V_k$. Then, for every $\phi \in \maF$, we have that
\begin{equation*}
  \Prim(A) \smallsetminus \supp(\phi) \, = \, 
  \{ I \in \Prim(A), \ker(\phi) \not\subset I \} \, = \,
  \Prim(\ker(\phi))
\end{equation*}
is an open subset of $\Prim(A)$ containing $I_0$,
and hence it is a neighborhood of $I_0$ in
$\Prim(A)$. Therefore there exists $n$ such that $V_n \subset \Prim(A)
\smallsetminus \supp(\phi)$ and hence $\|\pi(a)\| \le 1 - 2^{-n}$ for
all $\pi$ such that $\ker(\pi) \in \supp(\phi)$. This gives
$\|\phi(a)\| \le 1-2^{-n} < 1$, thus contradicting the fact that
$\maF$ is strictly norming. This proves the first half of the
statement.

Let us prove the converse, that is, the second half of the statement,
which is easier. Thus let us assume that every strictly norming family
of representations of $A$ is also exhausting and let us prove that
every primitive ideal $I_0 := \ker(\pi_0) \in \Prim(A)$ has a
countable basis for its system of neighborhoods. Let us fix then $I_0
:= \ker(\pi_0) \in \Prim(A)$ arbitrarily and show that it has a
countable basis for its system of neighborhoods.  Also, we associate
to each primitive ideal $I \in \Prim(A)$ an irreducible representation
$\phi_I$ with kernel $I$.  By remark \ref{rem.simple}, we
have that the family of representations $\maF := \{ \phi_I, I \in Prim(A), I \neq I_0\}$ 
is not exhausting, since  $\Prim(A)$ is a $T_1$ space (and hence
its points are closed) and hence $\supp{\phi_I} = I$. 
By our assumption, the family $\maF$  is hence
also not strictly norming. Therefore, by the definition of a strictly
norming family of representations, there exists $a \in A$, such that $\|\pi(a)\| < \|a\|$
for all irreducible $\pi$ with $\ker(\pi) \neq I_0$. Note that since
the family $\widehat{A}$ is strictly norming (see Example
\ref{ex.irred}), we have that $\|a\| = \max_{\pi \in \widehat{A}}
\|\pi(a)\|$, and hence $\|a\| = \|\pi_0(a)\|$. By rescaling, we can
assume $\|a\| = \|\pi_0(a)\| = 1$. Then the sets
\begin{equation*}
 V_n \, := \, \{ \, \ker(\pi) \in \Prim(A), \, \|\pi(a)\| > 1 - 2^{-n} \, \}
\end{equation*}
are open neighborhoods of $I_0 := \ker(\pi_0)$ in $\Prim(A)$ by Lemma
\ref{lemma.usc}. Let us show that they form a basis for the system of
neighborhoods of $I_0$.  Indeed, let $G$ be an arbitrary open subset
of $\Prim(A)$ containing $I_0$. Then there exists a two-sided ideal $J
\subset A$ such that $G = \Prim(J)$. The set of irreducible
representations of $A/J$ identifies with $\Prim(J)^c := \Prim(A)
\smallsetminus \Prim(J)$, and hence it does not contain $\pi_0$. Hence
$\|\pi(a)\|<1$ for all $\pi \in \Prim(A/J)$. Since $\widehat{A/J}$ is
a strictly norming family of representations of $A/J$, we obtain that
$\|a + J\|_{A/J} < 1$ (the norm is in $A/J$). Let $n$ be such that
$\|a + J \|_{A/J} \le 1 - 2^{-n}$. Then $V_n \subset \Prim(J) = G$,
which completes the proof of the second half of this theorem. The
proof is now complete.
\end{proof}

Clearly, there are $C\sp{\ast}$-algebras for which the spectrum is not 
$T_1$, but for which every strictly norming family of representations is
also exhausting. We do not know, however, if the converse result is true
in full generality (that is, for every $C\sp{\ast}$-algebra). 
It is easy to show that separable $C\sp{\ast}$-algebras satisfy the
assumptions of Proposition \ref{prop.2ndcount}.

\begin{theorem}\label{thm.2nd.count} Let $A$ be a 
  separable $C\sp{\ast}$-algebra. Then every primitive ideal $I \in
  \Prim(A)$ has a countable base for its system of
  neighborhoods. Consequently, if $\maF$ is a strictly norming set of
  representations of $A$, then $\maF$ is exhausting.
\end{theorem}

\begin{proof} It is known \cite{Dixmier} that $\Prim(A)$ is 
second countable. This gives the result in view of Proposition \ref{prop.2ndcount}.
For the benefit of the reader, we now provide a quick proof that every point in
$\Prim(A)$, for $A$ separable, has a countable base for its system of neighborhoods.
Indeed, we can replace $A$ with $A\sp{+}$ and thus assume 
that $A$ is unital.  Let $\{a_n\}$ be a dense subset of $A$ and fix
$I_0 := \ker(\pi_0) \in \Prim(A)$. Define
\begin{equation*}
	V_n \, := \, \{ \, \ker(\pi) \in \Prim(A), \ \|\pi(a_n)\| >
        \|\pi_0(a_n)\|/2 \, \} \, .
\end{equation*}
Then each $V_n$ is open by Lemma \ref{lemma.usc}. We claim that $V_n$
is a basis of the system of neighborhoods of $I_0 := \ker(\pi_0)$ in
$\Prim(A)$.  Indeed, let $G \subset \Prim(A)$ be an open set
containing $I_0$. Then $G = \Prim(J)$ for some two-sided ideal of $A$
such that $\pi_0 \neq 0$ on $J$.  Let $a \in J$ such that $\pi_0(a) \neq
0$. By the density of the sequence $a_n$ in $A$, we can find $n$ such
that $\|a - a_n\|< \|\pi_0(a)\|/4$.  Then $\|\pi'(a) - \pi'(a_n)\| <
\|\pi_0(a)\|/4$ for any irreducible representation $\pi'$, and hence
\begin{equation}\label{eq.ineq}
	\|\pi'(a)\| - \|\pi_0(a)\|/4 < \| \pi'(a_n) \| < \|\pi'(a)\| +
        \|\pi_0(a)\|/4 \, , \quad \forall \pi' \in \hat A \, .
\end{equation}
To show that $V_n \subset G$, it is enough to show that $V_n \cap G^c
= V_n \cap \Prim(J)^c = \emptyset$.  Suppose the contrary and let $\pi
\in \hat A$ be such that $\ker(\pi) \in V_n \cap \Prim(J)^c$.  Then
$\|\pi(a_n)\| > \|\pi_0(a_n)\|/2$, by the definition of
$V_n$. Moreover, $\pi(a) = 0$ since $a \in J$ and $\pi$ vanishes on
$J$. Let us show that this is not possible. Indeed, using Equation
\eqref{eq.ineq} twice, for $\pi ' = \pi_0$ and for $\pi' = \pi$, we
obtain
\begin{equation*}
	\frac{3}{8} \|\pi_0(a)\| \, < \, \frac{1}{2} \|\pi_0(a_n)\| \,
        < \, \|\pi(a_n)\| \, < \frac{1}{4} \|\pi_0(a)\| \,,
\end{equation*}
which is contradiction. Consequently $V_n \subset G$ and hence
$\{V_n\}$ is a basis for the system of neighborhoods of $\pi_0$ in
$\Prim(A)$, as claimed. The last part follows from the first 
part of Proposition \ref{prop.2ndcount}.
\end{proof}

The next two basic examples illustrate the differences between the
notions of faithful and strictly norming families.

\begin{example}\label{matrix}
Let in this example $A$ be the $C\sp{\ast}$-algebra of continuous
functions $f$ on $[0,1]$ with values in $M_2(\mathbb{C})$ such that
$f(1)$ is diagonal, which is a type I $C\sp{\ast}$-algebra, and thus
we identify $\hat A$ and $\Prim(A)$. Then the maps $\ev_t \colon
f\mapsto f(t) \in M_2(\CC)$, for $t<1$, together with the maps
$\ev_1^{i}\colon f \mapsto f(1)_{ii}$ ($i=0,1$) provide all the
irreducible representations of $A$ (up to equivalence). The family
\begin{equation*}
 \maF \, = \, \{\, ev_t,\,,t<1 \, \} \, \cup \, \{\ev_1^1 \, \}
\end{equation*}
is a faithful but not exhausting family. In fact the function
$t\mapsto {\scriptstyle \begin{pmatrix}1&0\\0&1-t\end{pmatrix}}$ is
not invertible in $A$ but $\pi(f)$ is invertible for all
$\pi\in\maF$. Of course, in this example, every $\pi \in \hat A =
\Prim(A)$ has a countable base for its system of neighborhoods, so
every strictly norming family of representations $\maF$ of $A$ is also
exhausting.
\end{example}

The next example is closely related to the examples we will be dealing
with below.
\begin{example} \label{toeplitz}
Let $\maT$ be the Toeplitz algebra, which is again a type I
$C\sp{\ast}$-algebra, and thus we again identify $\widehat \maT$ and
$\Prim(\maT)$. The Toeplitz algebra $\maT$ is defined as the
$C\sp{\ast}$-algebra generated by the operator defined by the
unilateral shift $S$. (Recall that $S$ acts on the Hilbert space
$L^2(\mathbb{N})$ by $S\colon \epsilon_k \mapsto \epsilon_{k+1}$.) As
$S^*S = 1$ and $SS^*-1$ is a rank $1$ operator, one can prove that the
following is an exact sequence
\begin{equation*} 
  0 \to \mathcal{K} \to \maT \to \maC(S^1)\to 0 \,,
\end{equation*} 
where $\mathcal{K}$ is the algebra of compact operator. Extend the
unique irreducible representation $\pi$ of $\mathcal{K}$ to $\maT$ as
in \cite{Dixmier}. Also, the irreducible characters $\chi_\theta$ of
$S^1$ pull-back to irreducible characters of $\maT$ vanishing on
$\mathcal{K}$. Then the spectrum of $\maT$ is
\begin{equation*} 
  \widehat{\maT} \, = \, \{\pi\}\cup
  \{\chi_\theta\,,\,\theta\in S^1\}\,,
\end{equation*} 
with $S^1$ embedded as a closed subset. A subset $V \subset
\Prim(\maT)$ will be open if, and only if, it contains $\pi$ and its
intersection with $S^1$ is open. We thus see that the single element
set $\{\pi\}$ defines an exhausting family. In other words
$\widehat{\maT}=\overline{\{\pi\}} = \supp(\pi)$. Since every
exhausting family is also strictly norming, by Proposition
\ref{prop.full.strictly.norming}, the family $\{\pi\}$ consisting of a
single representation is also strictly norming.
We can see also directly that the family $\maF = \{\pi\}$ (consisting
of $\pi$ alone) is strictly norming. Indeed, it suffices to notice that
$\|x\|=\|\pi(x)\|$ for all $x$ since $\pi$ is injective.
In this example again every $\pi\rp \in \hat \maT = \Prim(\maT)$ has a
countable base for its system of neighborhoods, so every strictly
norming family of representations $\maF$ of $\maT$ is also exhausting.
\end{example}

Here are two more examples that show that the condition that $A$ be
separable is not necessary for the classes of exhausting families of
representations and strictly norming families of representations to
coincide.
\begin{example} 
Let $I$ be an infinite {\em uncountable} set. We endow it with the
discrete topology. Then $A_0 := \maC_0(I)$ and $A_1 :=\maK(\ell^2(I))$
(the algebra of compact operators on $\ell^2(I)$) are not separable,
however, if $\maF$ is a strictly norming family of representations of
$A_i$, $i=0, 1$, then $\maF$ is also an exhausting family of
representations of $A_i$.
\end{example}

\section{Unbounded operators}

The results of the previous sections are are relevant often in
applications to unbounded operators, so we now extend Theorem
\ref{thm.full} to (possibly) unbounded operators affiliated to
$C\sp{\ast}$-algebras. We begin with an abstract setting.

\subsection{Abstract affiliated operators}

The notion of affiliated self-adjoint operator has been extensively
and successfully studied, see \cite{BaajJulg, DamakGeorgescu,
  GeorgescuIftimovici, Vassout, Woronowicz} for example. In the sequel
we will closely follow the definitions in \cite{GeorgescuIftimovici},
beginning with an abstract version of this notion. See
\cite{KaadLesch, Pal} for results on unbounded operators on Hilbert
modules \cite{ConnesBook, Lance, ManuilovTroitsky}.

\begin{definition}
Let $A$ be a $C\sp{\ast}$-algebra. An \emph{observable $T$ affiliated
  to $A$} is a morphism $\theta_T : \maC_0(\RR) \to A$ of
$C\sp{\ast}$-algebras. The observable $T$ is said to be \emph{strictly
  affiliated to $A$} if the space generated by elements of the form
$\theta_T(h)a$ ($a \in A$, $h \in \maC_0(\RR)$), is dense in $A$.
\end{definition}

As in the classical case, we now introduce the Cayley transform.  To
this end, let us notice that an observable affiliated to $A$ extends
to a morphism $\theta_T\sp{+}\colon C_0(\RR)^+ \to A^+$ (the algebra
obtained from $A$ by adjunction of a unit). If moreover $T$ is
strictly affiliated to $A$, then $\theta_T$ extends to a morphism from
$C_b(\RR)$ to the multiplier algebra of $A$, but we shall not need
this fact.

\begin{definition}
Let $T$ be an observable affiliated to $A$. The {\em Cayley transform}
$u_T\in A^+$ of $T$ is
\begin{equation}\label{eq.Cayley}
	u_T := \theta_T\sp{+}(h_0)\,, \quad h_0(z) := (z + \imath)(z -
        \imath)^{-1} \,.
\end{equation}
\end{definition}

The Cayley transform allows us to reduce questions about the spectrum
of an observable to questions about the spectrum of its Cayley
transform. Let us first introduce, however, the {\em spectrum of an
  affiliated observable.} Let thus $\theta_T : \maC_0(\RR) \to A$ be a
self-adjoint operator affiliated to a $C\sp{\ast}$-algebra $A$. The
kernel of $\theta_T$ is then of the form $\maC_0(U)$, for some open
subset of $\RR$. We define the spectrum $\Spec_A(T)$ as the complement
of $U$ in $\RR$. Explicitly,
\begin{equation} \label{eq.def.spectrum}
  \Spec_A(T) \, = \, \{\lambda\in\RR , \ h(\lambda) = 0, \ \forall
  h\in\maC_0(\RR) \, \mbox{ such that } \, \theta_T(h)=0\, \} \;.
\end{equation} 
We allow the case $\Spec_A(T) = \emptyset$, which corresponds to the
case $T = \infty$ or $u_T = 1$. If $\sigma : A \to B$ is a morphism of
$C\sp{\ast}$-algebras, then $\sigma \circ \theta_T : \maC_0(\RR) \to
A$ is an observable $\sigma(T)$ affiliated to the $C\sp{\ast}$-algebra
$B$ and
\begin{equation}
  \Spec_B(\sigma(T)) \, \subset \, \Spec_A(T) \;.
\end{equation}  
If $\sigma$ is injective, then $\Spec_B(\sigma(T)) = \Spec_A(T)$,
which shows that the spectrum is preserved by increasing the
$C\sp{\ast}$-algebra $A$. Note that
\begin{equation}
   \sigma(u_T) \, = \, u_{\sigma(T)} \;.
\end{equation}
By classical results, if $(u_T - 1)$ is injective, then we can define
a true self-adjoint operator $T := \imath (u_T + 1) (u_T - 1)^{-1} \in
A$ such that $\theta_T(h) = h(T)$, $h \in \maC_0(\RR)$
\cite{Derezinski-Gerard}.  This is the case, for instance, if If
$\Spec(T)$ is a bounded subset of $\RR$, in which case we shall say
that $T$ is {\em bounded}. In any case, bounded or unbounded, our
definition of $\Spec(T)$ in terms of $\theta_T$ coincides with the
classical spectrum of $T$ defined using the resolvent. Let $h_0(z) :=
(z + \imath)(z - \imath)^{-1}$, as before.

\begin{lemma}\label{lemma.s.inv}
The spectrum $\Spec(T)$ of the an observable $\theta_T :
\maC_0(\RR)\to A$ affiliated to the $C\sp{\ast}$-algebra $A$ and the
spectrum $\Spec(u_T)$ of its Cayley transform are related by
\begin{equation*}
  \Spec(T) \, = \, h_0^{-1}\left(\Spec(u_T) \right) \;.
\end{equation*}
\end{lemma}

\begin{proof}
This follows from the fact that $h_0$ is a homeomorphism of $\RR$ onto
its image in $S^1 := \{|z| = 1\}$ and from the properties of the
functional calculus.
\end{proof}

Let us notice that the above lemma is valid also in the case when
\begin{equation*}
  T = \infty \, \Leftrightarrow \, \theta_T = 0 \, \Leftrightarrow \,
  \Spec(T) = \emptyset \, \Leftrightarrow \, u_T = 1 \,
  \Leftrightarrow \, \sigma(u_T) = \{1\}\,.
\end{equation*}
One can make the relation in the above lemma more precise by saying
that, for bounded $T$, we have $h_0(\Spec(T))=\Spec(u_T)$, whereas for
unbounded $T$ we have
\begin{equation}
  \overline{ h_0(\Spec(T)) } = h_0(\Spec(T)) \cup \{1\} =
  \Spec(u_T) \,,
\end{equation}
where $h_0(z) := (z + \imath)(z - \imath)^{-1}$, as before.

Here is our main result on (possibly unbounded) self-adjoint operators
affiliated to $C\sp{\ast}$-algebras.

\begin{theorem}\label{thm.spectrum} 
Let $A$ be a unital $C\sp{\ast}$-algebra and $T$ an observable
affiliated to $A$.  Let $\maF$ be a set of representations of $A$.
\begin{enumerate}
\item If $\maF$ is strictly norming, then
\begin{equation*}
	\Spec(T) \, = \, \cup_{\phi \in \maF} \, \Spec(\phi(T)) \,.
\end{equation*}

\item If $\maF$ is faithful, then
\begin{equation*}
	\Spec(T) \, = \, \overline{\cup}_{\phi \in \maF} \,
        \Spec(\phi(T)) \,.
\end{equation*}
\end{enumerate}
\end{theorem}

\begin{proof} 
The proofs of (i) and (ii) are similar, starting with the relation
$\Spec(T) \, = \, h_0^{-1}(\Spec(u_T))$ of Lemma \ref{lemma.s.inv}. We
begin with (i), which is slightly easier. Since $\maF$ is strictly
norming, we can then apply theorem \ref{thm.full2} to $u_T\in A^+$ and
the family $\sigma\in\maF$. We obtain
\begin{multline*}
  \Spec(T) \, = \, h_0^{-1} \left [ \Spec(u_T) \right ] \, = \,
  h_0^{-1} \left [ \cup_{\sigma\in\maF}\, \Spec(\sigma(u_T)) \right ]
  \\
    \, = \, h_0^{-1}\left [ \cup_{\sigma\in\maF}\,
      \Spec(u_{\sigma(T)}) \right ]
    \, = \, \cup_{\sigma\in\maF}\, h_0^{-1} \left[
      \Spec(u_{\sigma(T)}) \right ]
    \, = \, \cup_{\sigma\in\maF}\, \Spec(\sigma(T)) \,.
\end{multline*}

If, on the other hand, $\maF$ is faithful, we apply proposition
\ref{prop.faithful4} after noting that $h_0$ is a homeomorphism of
$\RR$ onto its image in $S^1 := \{|z| = 1\}$ and hence
$h_0^{-1}(\overline{S}) = \overline{h_0^{-1}(S)}$ for any $S \subset
S^1$. The same argument then gives
\begin{multline*}
  \Spec(T) \, = \, h_0^{-1} \left [ \Spec(u_T) \right ] \, = \,
  h_0^{-1} \left [ \,\overline{\cup}_{\sigma\in\maF}\,
    \Spec(\sigma(u_T)) \right ] \\
    \, = \, h_0^{-1} \left [ \, \overline{\cup}_{\sigma\in\maF}\,
      \Spec(u_{\sigma(T)}) \right ]
    \, = \, \overline{\cup}_{\sigma\in\maF}\, h_0^{-1} \left[
      \Spec(u_{\sigma(T)}) \right ]
   \, = \, \overline{\cup}_{\sigma\in\maF}\, \Spec(\sigma(T)) \,.
\end{multline*} 
The proof is now complete.
\end{proof}

\begin{remark}\label{remark.spectrum}
In view of the remarks preceding it, Theorem \ref{thm.spectrum}
remains valid for true self-adjoint operators $T$.
\end{remark}

\subsection{The case of \lp true\rp\ operators}

We now look at concrete (true) operators.

\begin{definition}
Let $A\subset \mathcal{L}(\maH)$ be a sub-$C\sp{\ast}$-algebra of
$\mathcal{L}(\maH)$. A (possibly unbounded) self-adjoint operator $T :
\maD(T) \subset \maH \to \maH$ is said to be {\em affiliated} to $A$
if, for every continuous functions $h$ on the spectrum of $T$
vanishing at infinity, we have $h(T)\in A$.
\end{definition}
 
\begin{remark} 
We have that $T$ is affiliated to $A$ if, and only if,
$(T-\lambda)^{-1}\in A$ for one $\lambda \notin \Spec(T)$
(equivalently for all such $\lambda$) \cite{DamakGeorgescu}.  We thus
see that a self-adjoint operator $T$ affiliated to $A$ defines a
morphism $\theta_T : \maC_0(\RR) \to A$, $\theta_T(h) := h(T)$ such
that $\Spec(T) = \Spec(\theta_T)$. Thus $T$ defines an observable
affiliated to $A$.
\end{remark}

Since in our paper we shall consider only the case when
$A\subset\maL(\maH)$ is non degenerate, we shall not make a difference
between operators and observables affiliated to $A$.  Recall that an
unbounded operator $T$ is invertible if, and only if, it is bijective
and $T^{-1}$ is bounded. This is also equivalent to $0 \notin
\Spec(\theta_T)$. We have the following analog of Proposition
\ref{prop.faithful2} and Theorem \ref{thm.full}

\begin{theorem}\label{thm.spectrum2} 
Let $A \subset \maL(\maH)$ be a unital $C\sp{\ast}$-algebra and $T$ a
self-adjoint operator affiliated to $A$.  Let $\maF$ be a set of
representations of $A$.
\begin{enumerate}
\item Let $\maF$ be strictly norming. Then $T$ is invertible if, and
  only if $\phi(T)$ is invertible for all $\phi \in \maF$.

\item Let $\maF$ be faithful. Then $T$ is invertible if, and only if
  $\phi(T)$ is invertible for all $\phi \in \maF$ and the set $\{
  \|\phi(T)^{-1}\|, \phi \in \maF \}$ is bounded.
\end{enumerate}
\end{theorem}

\begin{proof}
This follows from Theorem \ref{thm.spectrum} as follows. First of all,
we have that $T$ is invertible if, and only if, $0 \notin
\Spec(T)$. Now, if $\maF$ is strictly norming, we have
\begin{equation*}
  0 \notin \Spec(T) \ \Leftrightarrow \ 0 \notin \cup_{\phi \in \maF}
  \, \Spec(\phi(T) ) \ \Leftrightarrow \ 0 \notin \Spec(\phi(T) )
  \ \mbox{ for all } \ \phi \in \maF \;.
\end{equation*}
This proves (i). (Note that $\phi(T)$ may not be a true operator, but
only and affiliated observable.)
To prove (ii), we proceed similarly, noticing also that the distance
from $0$ to $\Spec(T)$ is exactly $\|T^{-1}\|$.
\end{proof}

We have already remarked (Remark \ref{remark.spectrum}) that Theorem
\ref{thm.spectrum} extends to the framework of this subsection, that
is, that of (possibly unbounded) self-adjoint operators on a Hilbert
space.

\section{Parametric pseudodifferential operators}

Let $M$ be a compact smooth Riemannian manifold and $G$ be a Lie group
(finite dimensional) with Lie algebra $\mfkg :=
\operatorname{Lie}(G)$.  We let $G$ act by left translations on $M
\times G$. We denote by $\Psi^0(M \times G)^G$ the algebra of order
$0$, $G$-invariant pseudodifferential operators on $M \times G$ and
$\overline{\Psi^0(M \times G)^G}$ be its norm closure acting on $L^2(M
\times G)$. For any vector bundle $E$, we denote by $S^*E$ the set of
directions in its dual $E^*$. If $E$ is endowed with a metric, then
$S^*E$ can be identified with the set of unit vectors in $E^*$.  We
shall be interested the the quotient
\begin{equation*}
 S^*(T(M \times G))/G \, = \, S^*(TM \times TG)/G \, = \, S^*(TM
 \times \mfkg) \,.
\end{equation*}
We have that $\overline{\Psi^0(M \times G)^G} \simeq C_r^*(G) \otimes
\maK$ and then obtain the exact sequence
\begin{equation}\label{eq.Gexact}
  0 \, \to \, C_r^*(G) \otimes \maK \, \to \, \overline{\Psi^0(M
    \times G)^G} \, \to \, \maC(S^*(M \times \mfkg)) \, \to \, 0 \,,
\end{equation}
\cite{LMN1, LNGeometric, Monthubert, Vassout}.  Note that the kernel
of the symbol map will now have irreducible representations
parametrized by $\hat{G}_r$ the temperate unitary irreducible
representations of $G$. Let $T \in \Psi^m(M \times G)^G$ and denote by
$T\sp{\sharp} \in \Psi^m(M \times G)^G$ its formal adjoint (defined
using the calculus of pseudodifferential operators). All operators
considered below are closed with minimal domain (the closure of the
operators defined on $\CIc(M \times G)$).  We denote by $T\sp{*}$ the
Hilbert space adjoint of a (possibly unbounded) densely defined
operator.

\begin{lemma}\label{lemma.affiliated}
Let $T \in \Psi^m(M \times G)^G$ be elliptic. Then $T\sp{*} =
T\sp{\sharp}$.  Thus, if also $T = T\sp{\sharp}$, then $T$ is {\em
  self-adjoint} and $(T + \imath )\sp{-1} \in C_r^*(G)$, and hence it
is affiliated to $C_r^*(G)$.
\end{lemma}
\begin{proof} 
This is a consequence of the fact that $\Psi\sp{\infty}(M \times G)^G$
is closed under multiplication and formal adjoints.  See \cite{LMN1,
  LNGeometric, Vassout} for details.
\end{proof}

In other words, any elliptic, formally self-adjoint $T \in \Psi^m(M
\times G)^G$, $m > 0$, is actually self-adjoint.

Let us assume $G = \RR^n$, regarded as an abelian Lie group. Then our
exact sequence \eqref{eq.Gexact} becomes
\begin{equation}\label{eq.Gexact-bis}
  0 \to \maC_0(\RR^n) \otimes \maK \to \overline{\Psi^0(M \times
    \RR^n)^{\RR^n}} \to \maC(S^*(TM \times \RR^n )) \to 0 \,.
\end{equation}
This shows that $A := \overline{\Psi^0(M \times \RR^n)^{\RR^n}}$ is a
type I $C\sp{\ast}$-algebra, and hence we can identify $\hat A $ and
$\Prim(A)$.  Then we use that, to each $\lambda \in \RR^n$, there
corresponds an irreducible representation $\phi_\lambda$ of
$\maC_0(\RR^n) \otimes \maK$. Recalling that every irreducible
(bounded, *) representation of an ideal $I$ in a $C\sp{\ast}$-algebra
$A$ extends uniquely to a representation of $A$, we obtain that
$\phi_\lambda$ extends uniquely to an irreducible representation of
$\overline{\Psi^0(M \times \RR^n)^{\RR^n}}$ denoted with the same
letter. It is customary to denote by $\hat{T}(\lambda) :=
\phi_\lambda(T)$ for $T$ a pseudodifferential operator in $\Psi^m(M
\times \RR^n)^{\RR^n}$, $m \ge 0$. To define $\hat{T}(\lambda)$ for $m
> 0$, we can either use the Fourier transform or, notice that $\Delta$
is affiliated to the closure of $\Psi^0(M \times \RR^n)^{\RR^n}$. This
allows us to define $\hat {\Delta}(\lambda)$.  In general, we write $T
= (1 - \Delta)\sp{k} S$, with $S \in \Psi^0(M \times \RR^n)^{\RR^n}$
and define $\widehat T(\lambda)q = \widehat{(1 -
  \Delta)}(\lambda)\sp{k} \hat S(\lambda)$. (We consider the \dlp
analyst\rp s\drp\ Laplacian, so $\Delta \le 0$.)

\begin{lemma}\label{lemma.spectrum}
 Let $A := \overline{\Psi^0(M \times \RR^n)^{\RR^n}}$. Then the
 primitive ideal spectrum of $A$, $\Prim(A)$, is in a canonical
 bijection with the disjoint union $\RR\sp{n} \cup S^*(TM \times
 \RR^n)$, where the copy of $\RR\sp{n}$ corresponds to the open subset
 $\{\phi_\lambda, \lambda \in \RR\sp{n} \}$ and the copy of $S^*(TM
 \times \RR^n)$ corresponds to the closed subset $\{e_p, p \in S^*(TM
 \times \RR^n)\}$.  The induced topologies on $\RR\sp{n}$ and $S^*(TM
 \times \RR^n)$ are the standard ones.  Let $S^*M := S\sp{*}(TM)
 \subset S^*(TM \times \RR^n)$ correspond to $T\sp{*}M \subset
 T\sp{*}M \times \RR\sp{n}$. Then the closure of $\{\phi_{\lambda}\}$
 in $\Prim(A)$ is $\{\phi_{\lambda}\} \cup S\sp{*}M$.
\end{lemma}

\begin{proof} 
By standard properties of $C\sp{\ast}$-algebras (the definition of the
Jacobson topology), the ideal $\maC_0(\RR^n) \otimes \maK \subset A$
defines an open subset of $\Prim(A)$ with complement $\Prim(A/I)$ with
the induced topologies.  This proves the first part of the statement.

In order to determine the closure of $\{\phi_{\lambda}\}$ in
$\Prim(A)$, let us notice that the principal symbol of
$\hat{T}(\lambda)$ can be calculated in local coordinate carts on $M$
(more precisely, on sets of the form $U \times \RR\sp{n}$, with $U$ a
coordinate chart in $M$). This gives that the principal symbol of
$\hat{T}(\lambda)$ is given by the restriction of the principal symbol
of $T$ to $S\sp{*}M$.

Indeed, let $U = \RR\sp{k}$. A translation invariant
pseudodifferential $P$ operator on $U \times \RR\sp{n} = \RR\sp{k+n}$
is of the form $P = a(x, y, D_x, D_y)$ with $a$ independent of $y$:
$a(x, y, \xi, \eta) = \tilde{a}(x, \xi, \eta)$.  With this notation,
we have $\hat{P}(\lambda) = \tilde{a}(x, D_x, \lambda)$.  The
principals symbol of $\hat{P}(\lambda)$ is then the principal symbols
of the (global) symbol $\RR\sp{k} \ni (x, \xi) \to \tilde{a}(x, \xi,
\lambda)$, and is seen to be independent of the (finite) value of
$\lambda \in \RR\sp{n}$ and is the restriction from $S\sp{*}(TU \times
\RR\sp{n})$ to $S\sp{*}(TU \times \{0\})$ of the principal symbol of
$\tilde{a}$.

Returning to the general case, the same reasoning gives that the image
of $\phi_\lambda$ is $\overline{\Psi^0(M)}$. The primitive ideal
spectrum of this algebra is canonically homeomorphic to the closure of
$\{\phi_\lambda\}$, and this is enough to complete the proof.
\end{proof}

By the exact sequence \eqref{eq.Gexact}, in addition to the
irreducible representations $\phi_\lambda$, $\lambda \in \RR\sp{n}$
(or, more precisely, their kernels), $\Prim(A)$ contains also (the
kernels of) the irreducible representations $e_p(T) = \sigma_0(T)(p)$,
$p \in S^*(TM \times \RR^n)$.

\begin{proposition} Let $\maF:= \{\phi_\lambda, 
\lambda \in \RR^n \} \cup \{e_p, p \in S^*(TM \times \RR^n)
\smallsetminus S\sp{*}M \}$.
\begin{enumerate}[(i)]
\item The family $\maF$ is a strictly norming family of
  representations of $\overline{\Psi^0(M \times \RR^n)^{\RR^n}}$.
\item Let $P \in \Psi^m(M \times \RR^n)^{\RR^n}$, then $P: H^s(M
  \times \RR^n) \to H^{s-m} (M \times \RR^n)$ is invertible if, and
  only if $\hat{P}(\lambda) : H^s(M) \to H^{s-m}(M)$ is invertible for
  all $\lambda \in \RR^n$ and the principal symbol of $P$ is non-zero
  on all rays not intersecting $S\sp{*}M$.
\item If $T\in \Psi^m(M \times \RR^n)^{\RR^n}$, $m > 0$, is formally
  self-adjoint and elliptic, then $\Spec(e_p(T)) = \emptyset$,
  and hence
\begin{equation*}
	\Spec(T) \, = \, \cup_{\lambda \in \RR^n}\,
        \Spec(\hat{T}(\lambda)) \,.
\end{equation*}
\end{enumerate}
\end{proposition}

\begin{proof} (i) follows from Lemma \ref{lemma.spectrum}.
To prove (ii), let us denote by $\Delta_{M} \le 0$ the (non-positive)
Laplace operator on $M$.  Then the Laplace operator $\Delta$ on $M
\times \RR^n$ is $\Delta = \Delta_{\RR^n} + \Delta_{M}$.  Note that
$(1 - \Delta)^{m/2} : H^s(M \times \RR^n) \to H^{s-m} (M \times
\RR^n)$ and $(c - \Delta_{M})^{m/2} : H^s(M) \to H^{s-m} (M)$, $c >
0$, are isomorphisms. By \cite{LNGeometric}, we have that
\begin{equation*}
	P_1 \, := \, (1 - \Delta)^{(s-m)/2} P (1 - \Delta)^{- s/2} \in
        A := \overline{\Psi^0(M \times \RR^n)^{\RR^n}} \, .
\end{equation*}
It is then enough to prove that $P_1$ is invertible on $L^2(M \times
\RR^n)$.  Moreover from part (i) we have just proved and Theorem
\ref{thm.full} we know that $P_1$ is invertible on $L^2(M \times
\RR^n)$ if, and only if, $\hat{P}_1(\lambda) := \phi_\lambda(P_1)$ is
invertible on $L^2(M)$ for all $\lambda \in \RR^n$ and the principal
symbol of $P_1$ is non-zero on all rays not intersecting $S\sp{*}M$.
But, using also $\widehat{1 - \Delta}(\lambda) = (1 + |\lambda|^2 -
\Delta_{M})$, we have
\begin{equation*}
	\hat{P}_1(\lambda) = (1 + |\lambda|^2 - \Delta_{M})^{(s-m)/2}
        \hat P(\lambda) (1 + |\lambda|^2 - \Delta_{M})^{- s/2}\,,
\end{equation*}
which is invertible by assumption.
%

To prove (iii), we recall that $T$ is affiliated to $A$, by Lemma
\ref{lemma.affiliated}. The result then follows from Theorem
\ref{thm.spectrum}(1) (See also Remark \ref{remark.spectrum}).
\end{proof}

Operators of the kind considered in this subsection were used also in
\cite{ALN2, ConnesFol, DebordLescure, LeschPflaum, MonthubertPierrot,
  SchroheFrechet, SchSch1}. They turn out to be useful also for
general topological index theorems \cite{EmersonMeyerLast, NT2}. A
more class of operators than the ones considered in this subsection
were introduced in \cite{ASkandalis1, ASkandalis2}. The above result
has turned out to be useful for the study of layer potentials
\cite{QiaoNistor1}.  There are, of course, many other relevant
examples, but developing them would require too much additional
materials, so we plan to discuss these other examples somewhere else.

\def\cprime{$'$}

\end{document}